\documentclass[reqno,a4paper, 11pt]{amsart}

\usepackage[a4paper=true,pdfpagelabels]{hyperref}
\usepackage{graphicx}
\usepackage{latexsym}
\usepackage{mathabx}
\usepackage{amsmath}
\usepackage{amssymb}
\usepackage[T1]{fontenc}
\usepackage[utf8]{inputenc}
\usepackage{lmodern}

\usepackage{color}   

\newtheorem{theorem}{Theorem}[section]
\newtheorem{lemma}[theorem]{Lemma}
\newtheorem{corollary}[theorem]{Corollary}
\newtheorem{proposition}[theorem]{Proposition}

\newtheorem{lettertheorem}{Theorem}
\newtheorem{letterlemma}[lettertheorem]{Lemma}

\theoremstyle{definition}

\theoremstyle{remark}

\numberwithin{equation}{section}

\newcommand{\set}[1]{\left\{#1\right\}}
\newcommand{\abs}[1]{\left | #1\right |}
\newcommand{\nm}[1]{\left \| #1 \right \|}

\newcommand{\B}{\mathcal{B}}
\newcommand{\D}{\mathbb{D}}
\newcommand{\DD}{\widehat{\mathcal{D}}}
\newcommand{\Dd}{\widecheck{\mathcal{D}}}
\newcommand{\M}{\mathcal{M}}
\newcommand{\DDD}{\mathcal{D}}

\newcommand{\N}{\mathbb{N}}

\newcommand{\C}{\mathbb{C}}

\newcommand{\e}{\varepsilon}
\newcommand{\ep}{\varepsilon}

\renewcommand{\phi}{\varphi}

\newcommand{\T}{\mathbb{T}}

\def\VMOA{\mathord{\rm VMOA}}
\def\BMOA{\mathord{\rm BMOA}}
\newcommand{\mm}{\mathcal}

\def\a{\alpha}       \def\b{\beta}        \def\g{\gamma}
           \def\e{\varepsilon}
     \def\om{\omega}      
              
         \def\r{\rho}         \def\z{\zeta}
                  
\def\G{\Gamma}

\def\omg{\widehat{\omega}}

\def\mug{\widehat{\mu}}

\renewcommand{\H}{\mathcal{H}}

\newenvironment{Prf}{\noindent{\emph{Proof of}}}
{\hfill$\Box$ }

\begin{document}
	\title[Fractional Volterra-type operator acting on Hardy space]{Fractional Volterra-type operator induced by radial weight acting on Hardy space}

\author[C. Bellavita]{Carlo Bellavita}
\address{Departament of Matem\'atica i Inform\'atica, Universitat de Barcelona, Gran Via 585, 08007 Barcelona, Spain}
\email{carlo.bellavita@gmail.com}

\author[A. moreno]{\'Alvaro Miguel Moreno}
\address{Departamento de Analisis Matem\'atico, Universidad de M\'alaga, Campus de Teatinos, 
29071 Malaga, Spain}
\email{alvarommorenolopez@uma.es}

\author[G. Nikolaidis]{ Georgios Nikolaidis}
\address{Department of Mathematics, Aristotle University of Thessaloniki, 54124, Greece}
\email{nikolaidg@math.auth.gr}

\author[J. A. Pel\'aez]{Jos\'e \'Angel Pel\'aez}
\address{Departamento de Analisis Matem\'atico, Universidad de M\'alaga, Campus de Teatinos, 
29071 Malaga, Spain}
\email{japelaez@uma.es}

\thanks{The first author is member of Gruppo Nazionale per l’Analisi Matematica, la Probabilit\`a e le loro Applicazioni (GNAMPA) of Istituto Nazionale di Alta Matematica (INdAM) and he was supported by PID2021-123405NB-I00 by
the Ministerio de Ciencia e Innovaci\'on. The research of the second and fourth is supported in part by Ministerio de Ciencia e Innovaci\'on, Spain, project PID2022-136619NB-I00; La Junta de Andaluc{\'i}a,
project FQM210.
The third author is partially supported by the Hellenic Foundation for Research and Innovation (H.F.R.I.) under the ‘2nd Call for H.F.R.I. Research Projects to
support Faculty Members \& Researchers’ (Project Number: 4662).}

\subjclass{26A33, 30H10, 30H35, 47G10, 47B10}
\keywords{Radial doubling weight, fractional derivative, Volterra-type operator, Hardy space, BMOA}

\date{\today}
\begin{abstract}
Given  a radial doubling  weight $\mu$ on the unit disc $\mathbb{D}$ of the complex plane and its odd moments  $\mu_{2n+1}=\int_0^1 s^{2n+1}\mu(s)\, ds$, we consider the fractional derivative
	$$
	D^\mu(f)(z)=\sum_{n=0}^{\infty} \frac{\widehat{f}(n)}{\mu_{2n+1}}z^n, 
	$$
 of  a function $ f(z)=\sum_{n=0}^{\infty}\widehat{f}(n)z^n$ analytic in $\mathbb{D}$.

 We  also  consider the fractional integral operator
	$
	I^\mu(f)(z)=\sum_{n=0}^{\infty} \mu_{2n+1}\widehat{f}(n)z^n, 
	$ 
and the fractional  Volterra-type operator 
	$$
	V_{\mu,g}(f)(z)= I^\mu(f\cdot D^\mu(g))(z),\quad f\in\H(\D),
	$$
 for any fixed $g\in\H(\D)$.
We prove that $V_{\mu,g}$ is bounded (compact) on a Hardy space $H^p$, $0<p<\infty$, if and only  if $g$  belongs to $\BMOA$ ($\VMOA$).  Moreover, if 
$\int_0^1 \frac{\left(\int_r^1 \mu(s)\, ds\right)^p}{(1-r)^2}\,dr=+\infty$, we prove that  $V_{\mu,g}$ belongs to the Schatten class $S_p(H^2)$ if and only if $g=0$. On 
the  other hand,  if $\frac{\left(\int_r^1 \mu(s)\, ds\right)^p}{(1-r)^2}$ is a radial doubling weight it is proved that $V_{\mu,g} \in S_p(H^2)$ if and only if 
$g$ belongs to the  Besov space $B_p$.  En route, we obtain descriptions of $H^p$, $\BMOA$, $\VMOA$ and $B_p$ in terms of the fractional derivative $D^\mu$.

\end{abstract}
	\maketitle
	\section{Introduction}
Let $\H(\D)$ denote the space of analytic functions in the unit disc $\D=\{z\in\C:|z|<1\}$
and for a fixed function $g\in\H(\D)$ we consider the integral operator given by $$V_g(f)(z)=\int_0^z f(\z)g'(\z)\,d\z, \quad f\in \H(\D).$$
The Volterra-type operator $V_g$ is  prominent  in the study of   operator theory on spaces of analytic functions during the last decades.
 It probably has its origin in the bilinear operator $\left(f,g\right)\rightarrow \int f\,g'$  introduced by  A. Calder\'on in harmonic analysis in the $60$'s for his research on commutators of singular integral operators \cite{Calderon65}.
Pommerenke was  one of the first authors of the complex function theory community to consider the operator $V_g$ \cite{Pom}, and after the seminal works of Aleman, Cima and Siskakis \cite{AC,AS, ASBergman}, the study of the operator $V_ g$ on several spaces of analytic functions
has attracted a lot of attention in recent years  \cite{Aleman:Cascante:Fabrega:Pascuas:Pelaez,GLQ,MPPW,PR} in different contexts. Recently, different generalizations of $V_g$  defined in terms of  the $n$th derivative of the symbol $g$ and  iterated integrals have been considered in  \cite{ChaPAMS2020, DuLiQu2022}.

In this manuscript we are interested in proving that the classical results obtained in \cite{AC,AS, ASBergman},  on the boundedness, compactness and membership to Schatten classes of the  Volterra-type operator $V_g$ acting on Hardy and weighted Bergman spaces,
and the related recent ones obtained in \cite{ChaPAMS2020, ChaNik2025,DuLiQu2022},
 are examples of a general phenomenon  rather than  particular cases.  With this aim let us recall several concepts and introduce the required definitions. 

The extension 
$\mu(z)=\mu(|z|)$,\, $z\in\D$, of a non-negative function $\mu\in L^1([0,1))$  such that $\mug(r)=\int_{r}^1 \mu(s)\,ds>0$ for any $r\in(0,1)$, is called a radial weight. The 
 fractional derivative, induced by radial weight $\mu$, of  a function $ f(z)=\sum_{n=0}^{\infty}\widehat{f}(n)z^n \in \H(\D)$ is given by
	$$
	D^\mu(f)(z)=\sum_{n=0}^{\infty} \frac{\widehat{f}(n)}{\mu_{2n+1}}z^n, 
	$$
where $\mu_{x}=\int_0^1 s^{x}\mu(s)\, ds$, $x\ge 0$ are the moments of  $\mu$.
 We will also  consider the fractional integral operator
	$
	I^\mu(f)(z)=\sum_{n=0}^{\infty} \mu_{2n+1}\widehat{f}(n)z^n, 
	$ 
and the operator
	$$
	V_{\mu,g}(f)(z)= I^\mu(f\cdot D^\mu(g))(z),\quad f\in\H(\D),
	$$
 for any fixed $g\in\H(\D)$ and a radial weight $\mu$.

It is worth mentioning  that the fractional derivative $D^\mu$ was introduced in  \cite{MorenoPelRosa,PelRosa} and  if $\mu$ is a standard weight $\mu(z)=\beta(1-|z|^2)^{\beta-1}$, $\b >0$, then
$$D^{\mu}(f)(z)=D^{\beta}(f)(z)=\frac{2}{\G(\beta+1)}\sum\limits_{n=1}^{\infty} \frac{\G (n+\beta+1)}{\G(n+1)} \widehat{f}(n) z^n, 
$$
which basically coincides with the fractional derivative of order $\b>0$  introduced by Hardy and Littlewood in \cite[p. 409]{HLMathZ}.

 The Hardy space $H^p$ consists of $f\in\H(\D)$ for which
$
    \|f\|_{H^p}=\sup_{0<r<1}M_p(r,f)<\infty$,
where
    $$
    M_p(r,f)=\left (\frac{1}{2\pi}\int_0^{2\pi}
    |f(re^{i\theta})|^p\,d\theta\right )^{\frac{1}{p}},\quad 0<p<\infty,
    $$
and 
    $$
    M_\infty(r,f)=\max_{0\le\theta\le2\pi}|f(re^{i\theta})|.
    $$ 
It was  proved in \cite{AC,AS} that the classical Volterra-type operator $V_g$
is bounded on $H^p$, $0<p<\infty$, if and only if $g\in\BMOA$. 
The space $BMOA$ consists  of $f\in H^1$ such that the function of its boundary values $f(e^{i\theta})$ belongs to $\text{BMO}(\T)$.  The space $\BMOA$ can be equipped with different norms, we will write
\begin{equation}\label{eq:normaBMOA}
\nm{f}_{\BMOA}^2= \abs{f(0)}^2 + \sup_{a\in\D} \frac{\int_{S(a)}\abs{f'(z)}^2 (1-|z|^2)\,dA(z)}{1-\abs{a}}
\end{equation}
where  $dA(z)=\frac{dx\,dy}{\pi}$ is the normalized Lebesgue area measure on $\D$ and
  $$S(a)=\set{z\in\D: \abs{\arg{a}-\arg{z}}\leq \frac{1-\abs{a}}{2}, 1-\abs{z}\leq 1-\abs{a}},\quad  a\in \D\setminus\{0\},\quad  S(0)=\D,$$
is called the Carleson square induced by $a$. We recall the reader that a positive Borel measure $\mu$ on $\D$ is called a classical Carleson measure  if 
$\sup_{a\in\D}  \frac{\mu(S(a))}{1-\abs{a}}<\infty$, therefore $f\in \BMOA$ if and only if $|f'(z)|^2(1-|z|^2)\,dA(z)$ is a  classical Carleson measure.
Moreover, for any radial weight $\mu$ we consider the space $\BMOA_\mu$ of $g\in \H(\D)$ such that
 $$\|g\|^2_{\BMOA_\mu}=\sup_{a\in\D}\frac{\int_{S(a)}\abs{D^\mu(g)(z)}^2\frac{\mug(z)^2}{1-\abs{z}}dA(z)}{1-\abs{a}}<\infty.$$

\medskip It is clear that essential features of the fractional derivative operator $D^\mu$ are strongly affected  by the asymptotic behaviour of the moments of $\mu$, which is not well-understood for a general radial weight $\mu$
\cite{BonetLuskyTaskinen,DostanicIbero,PR19}. For this reason we will assume in many of our results that $\mu$ is a radial doubling weight. 
A radial weight $\mu$ belongs to $\DD$ if there exists $C=C(\mu)>0$ such that 
	$$
	\mug(r)\le C\mug\left(\frac{1+r}{2}\right),\quad 0\le r<1,
	$$
and $\mu\in\Dd$ if
	$$
	\mug(r)\ge C\mug\left(1-\frac{1-r}{K}\right),\quad 0\le r<1,
	$$ 
for some   $K=K(\mu)>1$ and $C=C(\mu)>1$.
Write $\mathcal{D}=\DD\cap\Dd$, and simply say that $\mu$ is a radial doubling weight if $\mu\in\DDD$.
The class $\DDD$ contains the standard radial weights and  $\mu\in \DDD$ if and only if for any (or some) $\beta>0$, 
$$\mu_x \asymp x^\beta  \left(\mu_{[\beta]} \right)_{x}, \quad x\ge 1,$$
where $\mu_{[\beta]}(s)=(1-s)^\beta \mu(s)$, see $(1.2)$ $(1.3)$ and Theorem~$3$ of \cite{PR19}. The class $\DDD$
 appear in a natural way in many questions on operator theory. For instance, the Bergman projection $P_\mu$ acts as a bounded and onto operator from the space $L^\infty$ of bounded complex-valued functions to the Bloch space $\mathcal{B}=\{ f \in \H(\D): \|f\|_{\B}=|f(0)|+ \sup_{z\in \D}(1-|z|^2)|f'(z)|<\infty\}$ if and only if $\mu\in\DDD$, and the Littlewood-Paley formula
	\begin{equation*}\label{LP}
	 \int_\D|f(z)|^p \mu(z)\,dA(z)\asymp \int_\D|f'(z)|^p(1-|z|)^p\mu(z)\,dA(z)+|f(0)|^p,\quad f\in \H(\D),
	\end{equation*}
is equivalent to $\mu\in\DDD$ \cite{PR19}.  For basic properties of the aforementioned classes of radial weights, concrete nontrivial
examples and more, see \cite{PelProcedings,PR,PR19} and the relevant references therein.

\begin{theorem}\label{th:Vmubounded}
Let  $0<p<\infty$, $\mu\in\DDD$ and $g\in\H(\D)$. Then,  the following conditions are equivalent:
\begin{enumerate}
\item[\rm(i)] $V_{\mu,g}$ is bounded on $H^p$; 
\item[\rm(ii)] $g\in \BMOA_\mu$;
 \item[\rm(iii)] $g\in \BMOA$.
\end{enumerate}
Moreover,
\begin{equation*}\label{eq:normVmug}
\| V_{\mu,g}\|_{H^p\to H^p}\asymp \|g\|_{\BMOA_\mu}\asymp \nm{g}_{\BMOA}.
\end{equation*}
\end{theorem}

We will also obtain a description of  the symbols $g\in \H(\D)$ such that  $V_{\mu,g}$ is compact on $H^p$.
Let us recall that a classical Carleson measure $\mu$ on $\D$ is a vanishing Carleson measure if  $\lim_{|a|\to 1^-}\frac{\mu(S(a))}{1-|a|}=0$. A function $g\in \H(\D)$ belongs to $\VMOA$ if  
 $|g'(z)|^2(1-|z|^2)\,dA(z)$ is a  vanishing  Carleson measure
and $g\in\VMOA_\mu$ if  $d \nu_g(z)=\abs{D^\mu(g)(z)}^2 \frac{\mug(z)^2}{1-\abs{z}}\, dA(z)$ is a vanishing Carleson measure.
\begin{theorem}\label{th:Vmucompact}
Let  $0<p<\infty$, $\mu\in\DDD$ and $g\in\H(\D)$. Then,  the following conditions are equivalent:
\begin{enumerate}
\item[\rm(i)] $V_{\mu,g}$ is compact on $H^p$; 
\item[\rm(ii)] $g\in \VMOA_\mu$;
 \item[\rm(iii)] $g\in \VMOA$.
\end{enumerate}
\end{theorem}

\medskip
The proofs of Theorems~\ref{th:Vmubounded} and \ref{th:Vmucompact} will be built on an ad hoc equivalent norm of  $H^p$ spaces in terms of the fractional derivative $D^\mu$, so that we can get rid of the fractional integral operator $  I^\mu$ which appears in the definition of $V_{\mu,g}$,  in the study of its boundedness and compactness.

It is well-known \cite{AhernBruna1988,FefSt} that for every $n\in\N$ and $0<p<\infty$ the Hardy space $H^p$ can be described in terms of the following Calderon's formula 
	\begin{equation}\label{eq: norma calderon}
		\nm{f}_{H^p}^p \asymp \sum_{j=0}^{n-1}\abs{f^{(j)}(0)}^p + \int_{\T} \left (\int_{\Gamma(\z)} \abs{f^{(n)}(z)}^2(1-\abs{z})^{2n-2}\,dA(z) \right )^{\frac{p}{2}} \abs{d\z}
	\end{equation}
where
	$
	\Gamma(\xi)=\left\{z\in\D:|\arg z-\arg\xi|<(1-|z|)\right\}
	$
is the non-tangential approach region (cone), with vertex at $\xi\in\T=\partial\D$.
We obtain the following extension of \eqref{eq: norma calderon}.
\begin{theorem}\label{th:main}
Let $\mu$ be a radial weight and $0<p<\infty$. Then, the following conditions are equivalent:
\begin{enumerate}
\item[\rm(i)] $\mu\in\DDD$;  
\item[\rm(ii)] $$
		\nm{f}^p_{H^p} \asymp \int_\T \left [ \int_{\Gamma(\xi)} \abs{D^\mu(f)(z)}^2 \left (\frac{\mug(z)}{1-\abs{z}}\right )^2 dA(z) \right ]^{\frac{p}{2}}\abs{d\xi},\quad f\in \H(\D);
		$$
\item[\rm(iii)] $$
		 \int_\T \left [ \int_{\Gamma(\xi)} \abs{D^\mu(f)(z)}^2 \left (\frac{\mug(z)}{1-\abs{z}}\right )^2 dA(z) \right ]^{\frac{p}{2}}\abs{d\xi}
 \lesssim \nm{f}^p_{H^p}, \quad f\in \H(\D).
		$$
\end{enumerate}
\end{theorem}

Theorem~\ref{th:main} follows from Theorem~\ref{th: norma equivalente <} and Theorem~\ref{th: norma equivalente >} below. In particular, the inequality
$$\nm{f}^p_{H^p} \lesssim \int_\T \left [ \int_{\Gamma(\xi)} \abs{D^\mu(f)(z)}^2 \left (\frac{\mug(z)}{1-\abs{z}}\right )^2 dA(z) \right ]^{\frac{p}{2}}\abs{d\xi},\quad f\in \H(\D),$$ holds for any radial weight $\mu$.
It is  worth mentioning that the main ideas in the proof of Theorem~\ref{th:main} comes from \cite{DuanRattyaWang2025}.

We will also obtain a description of $\BMOA$ in terms of the fractional derivative $D^\mu$.
\begin{theorem}\label{th: caracterizacion BMOAmain}
	Let $\mu$ be a radial weight. Then, the following conditions are equivalent:
\begin{enumerate}
\item[\rm(i)] $\mu\in\DDD$;  
\item[\rm(ii)] $$
	 \|g\|_{\BMOA_\mu} \asymp \nm{g}_{\BMOA},\quad g\in\H(\D);
	$$
\item[\rm(iii)] $$
	 \|g\|_{\BMOA_\mu} \lesssim \nm{g}_{\BMOA},\quad g \in\H(\D);
	$$
\item[\rm(iv)] $\VMOA_\mu= \VMOA$ and $$
	 \|g\|_{\BMOA_\mu} \asymp \nm{g}_{\BMOA},\quad g\in\VMOA;
	$$
\item[\rm(iv)] $\VMOA_\mu\subset \VMOA$ and $$
	 \|g\|_{\BMOA_\mu} \lesssim \nm{g}_{\BMOA},\quad g\in\VMOA.
	$$
\end{enumerate}
\end{theorem}
Theorem~\ref{th: caracterizacion BMOAmain} follows from Theorem~\ref{th: caracterizacion BMOA <} and  Theorem~\ref{th: caracterizacion BMOA >}. Especially, the quantitative embedding
$$ \nm{g}_{\BMOA}\lesssim \|g\|_{\BMOA_\mu}, \quad g\in\H(\D),$$
holds for any radial weight $\mu$. In the proofs of Theorem~\ref{th:main} and Theorem~\ref{th: caracterizacion BMOAmain} we will merge classical techniques employed on Hardy spaces with recent skills introduced in the context of Bergman spaces induced by radial weights. Let us recall that for $0<p<\infty$ and  a radial weight $\omega$, the Bergman  space $A^p_\om$ consists of functions $f\in \H(\D)$  such that
    $$
    \|f\|_{A^p_\omega}^p=\int_\D|f(z)|^p\omega(z)\,dA(z)<\infty.
    $$
In particular, it will be proved a  useful   identity 
which relates $D^\mu(f)$ with the $n$-th-derivative of $f$ (see  subsection~\ref{subse:22} below), which can be written either using the Taylor coefficient of $f$ or in terms of an integral on $\D$.   It will also be used
 precise Bergman norm estimates of the modified 
reproducing kernels $(1-\overline{z}\z)^NB^\om_z(\z)$ recently obtained in \cite[Lemma~5]{PelRatBMO}. Here  $B^\om_z$ is the   Bergman reproducing kernel  induced by a radial weight $\om$
at the point $z$. Theorem~\ref{th:main}, Theorem~\ref{th: caracterizacion BMOAmain} and a description of the positive Borel measures $\mu$ on $\D$
 such that the Hardy space is embedded in a tent space induced by $\mu$ (see Theorem~\ref{th:cohn} below) are the main ingredients in the proofs of  Theorem~\ref{th:Vmubounded} 
and Theorem~\ref{th:Vmucompact}. 

\medskip
Later on\textcolor{red}{,} we will prove the analogues to Theorem~\ref{th:Vmubounded} and Theorem~\ref{th:Vmucompact} on standard Bergman spaces $A^p_\alpha$, that is on $A^p_\om$ with
$\omega(z)=(\alpha+1)(1-|z|^2)^\alpha$, $\alpha>-1$. These results are (essentially) known due to the results in \cite{MorenoPelRosa, PelRosa} and we include them for the sake of completeness. 
 Therefore, let us recall that the little Bloch space
$\B_0$ consists of $g\in \B$ such that 
$\lim\limits_{\abs{z}\to 1^-}(1-\abs{z}^2)\abs{g'(z)} = 0$. Moreover, if $\mu$ is a radial weight we consider,
 the space $
    \mathcal{B}^\mu$ of $g\in\H(\D)$ such that $\nm{g}_{\mathcal{B}^\mu}=\sup_{z\in \D} \widehat{\mu}(z)\abs{D^\mu(g)(z)}<\infty$
and
$
\B_0^\mu = \set{f\in  \mathcal{B}^\mu : \lim\limits_{\abs{z}\to 1^-}\mug(z)\abs{D^{\mu}(f)(z)} = 0}.
$

\begin{theorem}\label{th:bergmanVmuboundedcompact}
	Let $\mu\in\DDD$, $0<p<\infty$, $\a>-1$ and $g\in\H(\D)$. Then,
	\begin{enumerate}
		\item[\rm(a)] The following conditions hold if and only if $V_{\mu,g}$ is bounded on $A^p_\a$;
\begin{enumerate}
\item[\rm(ai)]  $g\in\B^\mu$;
\item[\rm(aii)] $g\in\B$.
\end{enumerate}
Moreover, $$\| V_{\mu,g}\|_{A^p_\alpha\to A^p_\alpha}\asymp \|g\|_{\B^\mu}\asymp \nm{g}_{\B}.$$
	
\item[\rm(b)]  The following conditions hold if and only if $V_{\mu,g}$ is compact on $A^p_\a$;
\begin{enumerate}
\item[\rm(bi)]  $g\in\B_0^\mu$;
\item[\rm(bii)] $g\in\B_0$.
\end{enumerate}
\end{enumerate}
\end{theorem}

Our last main result concerns a description of the symbols $g\in \H(\D)$ such that  $V_{\mu,g}$, acting either on $H^2$ or on $A^2_\alpha$  ($\alpha>-1$), belongs to a
 Schatten  class. Therefore, we remind the reader that
for a compact operator $T$ on separable Hilbert space $\mm{H}$, there exists a decreasing sequence of positive numbers $\{\lambda_n(T)\}_n$ and orthonormal sets $\{e_n\}_n,\{\sigma_n\}_n\subset \mm{H}$ such that
$$T(x)=\sum_{n=1}^{+\infty}\lambda_n(T)\langle x,e_n\rangle \sigma_n\qquad x\in\mm{H}\,.$$
The sequence of numbers $\{\lambda_n(T)\}_n$ is called the  singular value sequence of $T$.
\par For $0<p<\infty$, the Schatten-von Neumann ideal of $\mm{H}$, denoted by $S_p(\mm{H})$, is the space of all compact operators $T$ acting on $\mm{H}$, such that the singular value sequence of $T$ belongs to $\ell^p$. Let
\begin{equation*}\label{norm of Schatten}
\|T\|_{S_p(\H)}=\left(\sum_{n=1}^{+\infty}\lambda_n(T)^p\right)^{\frac{1}{p}}\,.
\end{equation*}
On the other hand, for $0<p<\infty$ and $n_p$ the minimum $n\in \N$ such that $np>1$,  the classical Besov space $B_p$ consists of $f\in \H(\D)$ such that 
$$\| f\|^p_{B_p}=\sum_{j=0}^{n_p-1}|f(0)|^p+\int_{\D} |f^{(n_p)}(z)|^p(1-|z|^2)^{n_pp-2}\,dA(z)<\infty.$$
For $0<p<\infty$ and $\mu$ a radial weight such that $\frac{\mug(r)^p}{(1-r^2)^2}$ is a weight we introduce the fractional derivative Besov space $B_{p,\mu}$ of $f\in \H(\D)$ such that 
$$\| f\|^p_{B_{p,\mu}}=\int_{\D} | D^{\mu}(f)(z)|^p \frac{\mug(z)^p}{(1-|z|^2)^2}  \,dA(z)<\infty.$$
Here and  on the following we denote $A^2_{-1}=H^2$.

\begin{theorem}\label{th:schattendugintro}
Let $\mu\in \DDD$, $\alpha\ge -1$, $0<p<\infty$ and $g\in \H(\D)$. Then,
\begin{enumerate}
\item[\rm(a)] If $\int_0^1 \frac{\mug(r)^p}{(1-r)^2}\,dr=+\infty$, then $V_{\mu,g}\in S_p(A^2_\alpha)$ if and only if $g=0$.
 \item[\rm(b)] If $\frac{\mug(r)^p}{(1-r)^2}$ is a radial weight, then the following conditions are equivalent:
\begin{enumerate}
\item[\rm(i)] $V_{\mu,g}\in S_p(A^2_\alpha)$;
\item[\rm(ii)]  $g\in B_{p,\mu}$;
\end{enumerate} 
and  $$\| V_{\mu,g}\|_{S_p(A^2_\alpha)}\asymp \|g\|_{B_{p,\mu}}.$$ 

Moreover, if $\frac{\mug(r)^p}{(1-r)^2}\in \Dd$, (i) and (ii) are equivalent to the condition $g\in B_p$ and 
 $$\| V_{\mu,g}\|_{S_p(A^2_\alpha)}\asymp \|g\|_{B_{p,\mu}}\asymp \|g\|_{B_p}.$$ 

\end{enumerate}
\end{theorem}
It is worth mentioning that for any $1<p<\infty$,  the classical Volterra-type operator $V_g\in S_p(A^2_\a)$, $\a\geq-1$, if and only if $g\in \B_p$. In the case $0<p\le 1$,
 $V_g\in S_p(A^2_\a)$ if and only if $g$ is constant \cite{AS,ASBergman}. Going further,  we  consider the operator 
$$I^n(f)(z)=\int_0^z I_{n-1}(f)(z)\, dz,\quad I_0(f)(z)=z, \quad n\in \N,$$
and the operator  $V_{n,g}(f)(z) = I^n(f \cdot g^{(n)})(z)$. It was shown in \cite[Corollary~1, p.27]{DuLiQu2022} that for
any $\frac{1}{n}<p<\infty$,   $V_{n,g}\in S_p(A^2_\a)$, $\a>-1$, if and only if $g\in \B_p$. In the case $0<p\le \frac{1}{n}$,
 $V_g\in S_p(A^2_\a)$ if and only if $g^{(n-1)}$ is constant. Consequently, in view of all the previous results,  it is not surprising that Theorem~\ref{th:schattendugintro} states 
that  the cut-off for which there exist non-trivial operators $V_{\mu,g} \in S_p(A^2_\alpha)$ depends  on the condition  $\int_0^1 \frac{\mug(r)^p}{(1-r)^2}\,dr<+\infty$. 
 Indeed in  our framework, the analogue of the $n$-th derivative corresponds to the fractional derivative induced by the standard weight $\mu(r) = n(1 - r^2)^{n-1}$. In this case, $\frac{\widehat{\mu}(r)^p}{(1 - r)^2}$ is not a weight if and only if $p \leq \frac{1}{n}$. Therefore,  we rediscover in Theorem~\ref{th:schattendugintro}
 (as  particular cases and in certain sense) some of the results in \cite{AS,ASBergman,DuLiQu2022}.    
We also observe
that fixed $\mu \in \DDD$ and $p$ sufficiently large, Lemma~\ref{lemma: caract dcheck}~(ii) (below) ensures that there always exist non-trivial symbols $g$ such that $V_{\mu,g} \in S_p(A^2_\alpha)$.

Theorem~\ref{th:schattendugintro} follows from Theorem~\ref{th:schattendug} and Theorem~\ref{th:BpmuD} below. In Theorem~\ref{th:schattendug} it is proved that 
(i) and (ii) of  Theorem~\ref{th:schattendugintro} are equivalent. In fact, it is used  
the min-max theorem of eigenvalues for positive operators to connect the  singular value sequence of $ V_{\mu,g}$  with the  singular value sequence of a Toeplitz operator, and later on
a classical result due to H. Luecking \cite{Lu87} is employed to prove that (i) is equivalent to a discrete quantity equivalent to the $\| g\|^p_{ B_{p,\mu}}$. On the other hand, 
 Theorem~\ref{th:BpmuD} says that for any radial weight and $1<p<\infty$,
$B_p=B_{p,\mu}$ if and only $\frac{\mug(r)^p}{(1-r)^2}\in \DDD$, and the result remains true if  $0<p\le 1$ and $\mu\in \DD$. The proof Theorem~\ref{th:BpmuD} follows ideas from the proofs of  Theorem~\ref{th:main} and Theorem~\ref{th: caracterizacion BMOAmain} but a good number of extra obstacles have to be solved. In particular, different approaches are needed to deal with the cases $1<p<\infty$ and $0<p\le 1$.

The rest of the paper is organized as follows. Section~\ref{s2} is devoted to prove some technical results on radial weights  and useful representation formulas for the fractional derivative. Theorem~\ref{th:main} is proved in Section~\ref{s3} and Theorem~\ref{th: caracterizacion BMOAmain} is proved in Section~\ref{s4}. 
Section~\ref{s5} contains  proofs of Theorem~\ref{th:Vmubounded} and  Theorem~\ref{th:Vmucompact}. A proof of Theorem~\ref{th:bergmanVmuboundedcompact} is given in 
Section~\ref{s6}. Section~\ref{s7} contains results about embeddings between the spaces   $B_p$ and $B_{p,\mu}$. Finally,   Theorem~\ref{th:schattendugintro} is proved in Section~\ref{s8}.

To this end a couple of words about the notation used. The letter $C=C(\cdot)$ will denote an absolute constant whose value depends on the parameters indicated in the parenthesis, and may change from one occurrence to another.
We will use the notation $a\lesssim b$ if there exists a constant
$C=C(\cdot)>0$ such that $a\le Cb$, and $a\gtrsim b$ is understood
in an analogous manner. In particular, if $a\lesssim b$ and
$a\gtrsim b$, then we write $a\asymp b$ and say that $a$ and $b$ are comparable.

	\section{Previous results}\label{s2}
\subsection{Background on radial weights}
	Throughout the manuscript we will employ different descriptions of the classes of radial weights
	$\DD$ and $\Dd$. The next result gathers several characterizations of $\DD$ proved in \cite[Lemma 2.1]{PelProcedings}. 
	\begin{letterlemma}\label{lemma: caract Dgorro}
		Let $\om$ be a radial weight. Then, the following statements are equivalent:
		\begin{itemize}
			\item[(i)] $\om\in\DD$;
			\item[(ii)] There exists a constant $C=C(\om)>0$ such that
			$$
			\om_x \leq C\omg\left (1-\frac{1}{x} \right),\quad x\geq 1;
			$$
			\item[(iii)] There exists a constant $C=C(\om)>0$ such that
			$$
			\om_x\leq C\om_{2x},\quad x\geq1.
			$$
		\end{itemize}
	\end{letterlemma}
	We will also use a couple of descriptions of the class $\Dd$.
	\begin{letterlemma}\label{lemma: caract dcheck}
		Let $\om$ be a radial weight. Then, the following statements are equivalent: 
		\begin{itemize}
			\item[(i)] $\om\in\Dd$;
			\item[(ii)] There exists $C=C(\om)>0$ and $\b=\b(\om)>0$ such that
			$$
			\omg(s)\leq C \left ( \frac{1-s}{1-t} \right )^\b \omg(t),\quad 0\leq t \leq s <1;
			$$
\item[(iii)] For each (equivalently for some) $\gamma>0$ there is $\eta_0=\eta_0(\gamma)< 1$ such that for any  $\eta_0<\eta\leq 1$ there is $C=C(\om,\gamma, \eta)>0$ such that
			$$
			\int_0^r \frac{ds}{\omg(s)^\gamma(1-s)^\eta} \leq \frac{C}{\omg(r)^\gamma(1-r)^{\eta-1}},\quad 0\leq r < 1.
			$$  
		\end{itemize}
	\end{letterlemma}
	\begin{proof}
		The equivalence between (i) and (ii) is proved in \cite[Lemma B]{PelRosa}.
  Let us prove that (ii)$\Rightarrow$(iii). Take $\g>0$ and choose $\eta_0=1-\b\gamma$. Then,  if $\eta_0<\eta\leq 1$
		$$
		\int_0^r\frac{ds}{\omg(s)^\gamma(1-s)^\eta} \leq C^\gamma \frac{(1-r)^{\gamma\b}}{\omg(r)^\gamma}\int_0^r\frac{ds}{(1-s)^{\eta+\gamma\b}} \leq \frac{C^\gamma}{\eta+\gamma\b-1}\frac{1}{\omg(r)^\gamma(1-r)^{\eta-1}},\quad 0\leq r < 1,
		$$
so (iii) holds with $C(\om,\gamma, \eta)=\frac{C^\gamma}{\eta+\gamma\b-1}$.

		Conversely assume that (iii) holds for some $\gamma>0$. Then,  
		$$
		\frac{C}{\omg(t)^\gamma} \geq \int_0^t \frac{ds}{\omg(s)^\gamma (1-s)} \geq \int_r^t \frac{ds}{\omg(s)^\gamma(1-s)} \geq \frac{1}{\omg(r)^\gamma}\log\frac{1-r}{1-t},\, 0\leq r<t <1,
		$$
where  $C=C(\om,\gamma, 1)=C(\om,\gamma)$.
		Take  $K>1$ and $t=1-\frac{1-r}{K}$, then it follows  that 
		$$
		\frac{C}{\omg\left( 1-\frac{1-r}{K}\right)^\gamma}\geq \frac{1}{\omg(r)^\gamma} \log K,$$
that is $$\omg(r)\ge \left(\frac{\log K}{C} \right)^{1/\gamma}\omg\left( 1-\frac{1-r}{K}\right), \quad 0\le r<1.$$
So, choosing $K>e^C$ we deduce that $\om\in\Dd$. This finishes the proof.
		
	\end{proof}
	\subsection{Reproducing formulas and useful identities}\label{subse:22}
The   Bergman reproducing kernel  induced by a radial weight $\om$, can be written as 
 $$B^\om_z(\z)=\sum \overline{e_n(z)}e_n(\z),\quad z,\z\in\D,$$   for any orthonormal basis $\{e_n\}$ of $A^2_\om$. Therefore 
 using the basis
 induced by the normalized monomials
\begin{equation}\label{eq:B}
B^\om_z(\z)=\sum_{n=0}^\infty\frac{\left(\overline{z}\z\right)^n}{2\om_{2n+1}}, \quad z,\z\in \D.
\end{equation}

Now let us consider an integral-derivative fractional operator which includes the fractional derivative $D^\mu$ as a particular case. Indeed, 
	for a pair of radial weights $\om$ and $\nu$ the  integral-derivative fractional operator
	$$
	R^{\om,\nu}(f)(z)=\sum_{n=0}^{\infty}\frac{\om_{2n+1}}{\nu_{2n+1}}\widehat{f}(n)z^n,\quad f\in\H(\D),
	$$
	was first introduced in \cite{PeralaFraccionaria}, and  a calculation shows (see \cite{PeralaFraccionaria,PeralaRattyaFraccionaria}) that
\begin{equation}\label{eq:mu+}
   D^\mu(f)(z) = R^{1,\mu_+}(f)(z),\quad f\in\H(\D),\quad z\in\D,
\end{equation}
where  $\mu_+(z)=\int_{\abs{z}}^1\frac{\mu(s)}{s}ds,\, z\in\D\setminus\{0\}$. Now, we recall the reader
the following integral representation of $R^{\om,\nu}$  \cite[(3.3)]{PeralaFraccionaria}.
	\begin{lettertheorem}\label{th:perala}
		Let $\om$ and $\nu$ be radial weights.  Then,
		$$
		R^{\om,\nu}(f)(z) = \int_{\D} f(\z)\overline{B^{\nu}_{z}(\z)}\om(\z)dA(\z),\quad f\in A^1_\om,\quad z\in\D.
		$$
	\end{lettertheorem}
	From the above integral representation we can obtain an alternative formula for $R^{\om,\nu}$ in terms of the $n$th-derivative of the function.  In order to obtain that expression we introduce the following concepts.
	For a radial weight $\nu$ we consider the iterated radial weights 
	$$
	V_{\nu,n}(z)= 2\int_{\abs{z}}^1 rV_{\nu,n-1}(r)dr,\quad z\in \D,\, n\in\N\quad \text{and} \quad V_{\nu,0}(z)=\nu(z),
	$$
	and for a radial weight $\om$ we consider
	$\om^\ast(z)=\int_{\abs{z}}^1s\om(s)\log\frac{s}{\abs{z}}ds$ for $z\in\D\setminus\{0\}$ and its iterations
	$$
	\om^{\ast,n}(z)=\int_{\abs{z}}^1s\om^{\ast,n-1}(s)\log\frac{s}{\abs{z}}ds,\quad z\in \D\setminus\{0\},n\in \N \quad \text{and} \quad \om^{\ast,0}(z)=\om(z).
	$$ 
	\begin{proposition}\label{prop: expresión derivada fraccionaria general}
		Let  be $\om$ and $\nu$ be radial weights. Then, 
		\begin{itemize}
		\item[(i)] If $x>0$ and $n\in\N$ then $(V_{\nu,n})_x = \frac{2(V_{\nu,n-1})_{x+2}}{x+1}$ and $(B^{\nu}_z)^{(n)}(\z) = \overline{z}^n B_z^{V_{\nu,n}}(\z)$ for every $z,\z\in\D$;
		\item[(ii)] If $x>0$ and $n\in \N$ then $(\om^{\ast,n})_x = \frac{(\om_{\ast,n-1})_{x+2}}{(x+1)^2}$;
		\item[(iii)] Let $f\in A^1_\om$  and $n\in\N$,  then
		\begin{equation*}
			\begin{split}
				R^{\om,\nu}(f)(z) &= 
 \sum_{j=0}^{n-1}\frac{\om_{2j+1}}{\nu_{2j+1}}\widehat{f}(j)z^j + 4^nz^n\int_\D f^{(n)}(\z)\overline{B^{V_{\nu,n}}_z(\z)}\om^{\ast,n}(\z)dA(\z) \\
				&= \sum_{j=0}^{n-1}\frac{\om_{2j+1}}{\nu_{2j+1}}\widehat{f}(j)z^j + 4^nz^n R^{\om^{\ast,n},V_{\nu,n}}(f^{(n)})(z), \quad z\in\D.
			\end{split}
		\end{equation*}
		\end{itemize}
	\end{proposition}
	\begin{proof}
		The first equality  of (i) and (ii) can be obtained using  Fubini's theorem. To prove the second identity in (i), from the equality $(V_{\nu,n+1})_{2k+1}=\frac{(V_{\nu,n})_{2k+3}}{k+1}$ for every  $n,k\in \N\cup\{0\}$ we get
		$$
		(B^{V_{\nu,n}}_z)'(\z)=\overline{z} \sum_{k=0}^{\infty}\frac{(k+1)(\overline{z}\z)^k}{2(V_{\nu,n})_{2k+3}} = \overline{z}\sum_{k=0}^\infty \frac{(\overline{z}\z)^k}{2(V_{\nu,n+1})_{2k+1}}=\overline{z}B^{V_{\nu,n+1}}_z(\z),\quad z,\z\in\D,\quad n\in \N\cup\{0\}.
		$$
		From this identity we can prove the statement by induction. The case $n=1$ follows straightforward from the above identity with $n=0$. Next assume that $(B^{\nu}_z)^{(n)}(\z) = \overline{z}^n B^{V_{\nu,n}}_z(\z)$, for $n\in \N$. Then, 
		$$
		(B^\nu_z)^{(n+1)}(\z) = ((B^\nu_z)^{(n)})'(\z) = \overline{z}^n (B^{V_{\nu,n}}_z)'(\z) = \overline{z}^{n+1}B^{V_{\nu,n+1}}(\z),\quad z,\z\in\D.
		$$
Let us prove (iii).  Arguing as in the proof of \cite[(4.3)]{PR} for any $f,g\in \H(\D)$ such that $f\overline{g}\in L^1_\om$, we get
\begin{equation}\begin{split}\label{eq:j1}
\langle f,g \rangle_{A^2_\om} & =\int_{\D} f(\z)\overline{g(\z)}\, \om(\z)\,dA(\z)=\int_{\D} f(\z)\overline{g(\z)}\, \om^{\ast,0}(\z)\,dA(\z)
\\ & = \widehat{f}(0)\overline{\widehat{g}(0)}(\om^{\ast,0})_1+4\int_{\D} f'(\z)\overline{g'(\z)}\, \om^{\ast,1}(\z)\,dA(\z)
\\ & = \sum_{j=0}^{n-1} 4^j\widehat{f^{(j)}}(0)\overline{\widehat{g^{(j)}}(0)}(\om^{\ast,j})_1+4^n\int_{\D} f^{(n)}(\z)\overline{g^{(n)}(\z)}\, \om^{\ast,n}(\z)\,dA(\z).
\end{split}\end{equation}
So, taking $g=B^{\nu}_z$ and $f\in A^1_\om$, Theorem~\ref{th:perala}, the above identity and part (i) yield
\begin{equation}\begin{split}\label{eq:j2}
R^{\om,\nu}(f)(z) &= \langle f, B^{\nu}_z\rangle_{A^2_\om}
\\ & = \sum_{j=0}^{n-1} 4^j\widehat{f^{(j)}}(0)\overline{\widehat{(B^{\nu}_z)^{(j)}}(0)}(\om^{\ast,j})_1+4^n\int_{\D} f^{(n)}(\z)\overline{(B^{\nu}_z)^{(n)}(\z)}\, \om^{\ast,n}(\z)\,dA(\z)
\\ & =  \sum_{j=0}^{n-1} 4^j\widehat{f^{(j)}}(0)\overline{\widehat{(B^{\nu}_z)^{(j)}}(0)}(\om^{\ast,j})_1
 +4^nz^n R^{\om^{\ast,n},V_{\nu,n}}(f^{(n)})(z), \quad z\in\D.
\end{split}\end{equation}
Next, bearing in mind part (i)-(ii) it follows that
\begin{equation}\begin{split}\label{eq:j3}
\widehat{f^{(j)}}(0) &=j! \widehat{f}(j)
\\ (V_{\nu,j})_1 &=\frac{\nu_{2j+1}}{j!}
\\ \widehat{(B^{\nu}_z)^{(j)}}(0) & =\overline{z}^j\widehat{B^{V_{\nu,j}}_z}=\frac{\overline{z}^j}{ (V_{\nu,j})_1}=\frac{j! \overline{z}^j}{ \nu_{2j+1}}
\\ (\om^{\ast,j})_1 & = \frac{\om_{2j+1}}{4^j (j!)^2}.
\end{split}\end{equation}
Therefore, joining \eqref{eq:j2} and  \eqref{eq:j3} we prove (iii). This finishes the proof.

	\end{proof}

	In the particular case $\om=1$, we will write  
$W^n = \om^{\ast,n}, \, n\in \N\cup\{0\}.$
 Then, 
joining \eqref{eq:mu+} and  Proposition~\ref{prop: expresión derivada fraccionaria general} we obtain the following formula, which  is instrumental in the proofs of some of the main results of the manuscript, and relates the fractional derivative $D^\mu(f)$ with the $n$th-derivative of $f$.

	\begin{corollary}\label{prop: expresión derivada fraccionaria}
		Let $\mu$ be a radial weight, then 
		$$
		D^{\mu}(f)(z) = \sum_{j=0}^{n-1} \frac{\widehat{f}(j)}{\mu_{2j+1}}z^j + 4^n z^n R^{W^n,V_{\mu_+,n}}(f^{(n)})(z),\quad f\in A^1,z\in\D,n\in \N.
		$$
	\end{corollary}
	The following is a technical lemma that will be useful for our purposes.

	\begin{lemma}\label{lema pesos tec}
		Let $\mu$ be a radial weight. Then,
		\begin{itemize}
			\item[(i)] If $n\in\N$, then 
			\begin{equation}\label{eq: lema pesos cond 1}
				W^n(r) \asymp (1-r)^{2n},\quad \frac{1}{2}\leq r<1;
			\end{equation}
			\item[(ii)]
			$$
			\widehat{\mu_+}(r) \leq\mug(r) (1-r),\quad 0\leq r <1;
			$$
			\item[(iii)] If $n\in\N$, then
			\begin{equation}\label{eq: lema pesos cond 3}
			V_{\mu_+,n}(r) \lesssim \mug(r)(1-r)^n,\quad 0\leq r <1;
			\end{equation}
			\item[(iv)] If $\mu\in\DD$ ,then 
			$$
			\widehat{\mu_+}(r) \gtrsim \mug(r)(1-r),\quad 0\leq r <1;
			$$
			\item[(v)] If $\mu\in \DD$ and $n\in\N$, then
			\begin{equation}\label{eq: lema pesos cond 5}
				\widehat{V_{\mu_+,n}}(r) \gtrsim \mug(r)(1-r)^{n+1},\quad \frac{1}{2}\leq r <1.
			\end{equation}
		\end{itemize}
	\end{lemma}
		\begin{proof}
			Since
			$
			(s-r) \leq s \log\frac{s}{r}\leq \frac{s}{r}(s-r)\leq \frac{1}{r}(s-r),\, 0<r\leq s < 1,\,$ then
			\begin{equation}\label{eq: lema pesos ast}
				W^n(r) \asymp \int_r^1 (s-r)W^{n-1}(s)ds, \quad \frac{1}{2}\le r<1, \, n\in\N.
			\end{equation}
Threfore, we get (i) by induction. 
Let us prove  (ii) and (iv).
			Let $0\leq r <1$, then by Fubini's theorem
			$$
			\widehat{\mu_+}(r) = \int_r^1 \int_s^1 \frac{\mu(t)}{t}dt \, ds = \int_r^1 \mu(t) \frac{t-r}{t}\, dt.
			$$
			Fixed $0\leq r <1$, the function $h_r(t)=\frac{t-r}{t}$ is non-decreasing in the interval $(0,1]$, so 
			$$
			\widehat{\mu_+}(r)\le
\mug(r)(1-r), \quad 0\le r<1,
			$$
			obtaining (i).
			On the other hand, by applying that $\mu\in \DD$ and that $h_r$ is non-decreasing
			\begin{equation*}
				\begin{split}
					\widehat{\mu_+}(r) &= \int_r^1 \mu(t) \frac{t-r}{t}\,dt 
                          \geq \int_{\frac{1+r}{2}}^1 \mu(t) \frac{t-r}{t}dt 
                         \\ & \geq \mug\left (\frac{1+r}{2} \right )\frac{1-r}{1+r} \gtrsim \mug(r)(1-r), \quad 0\le r<1,
				\end{split}
			\end{equation*}
			concluding the proof for (ii) and (iv). \\
			To prove (iii) we work by induction. If $n=1$, then by (ii)
			$$
			V_{\mu_+,1}(r) =  2\int_{r}^1s\mu_+(s)ds\ \lesssim \widehat{\mu_+}(r) \leq \mug(r)(1-r),\quad 0\leq r <1.
			$$
			Now assume that \eqref{eq: lema pesos cond 3} holds for $n\in\N$. Then, 
			\begin{equation*}
				\begin{split}
					V_{\mu_+,n+1}(r)& 
\lesssim \int_{r}^1 \mug(s)(1-s)^n  ds 
					\leq 
 \mug(r)(1-r)^{n+1},\quad 0\leq r <1.  
				\end{split}
			\end{equation*}
			Lastly, we are going to prove (v). If $n=1$ using  (iv) and the hypotheses $\mu\in\DD$
			\begin{equation*}
				\begin{split}
					\widehat{V_{\mu_+,1}}(r) & 
\geq \int_r^1 \widehat{\mu_+}(s)ds \gtrsim 
\int_r^{\frac{1+r}{2}} \mug(s)(1-s)ds
\gtrsim \mug(r)(1-r)^2, \quad \frac{1}{2}\leq r <1.
				\end{split}
			\end{equation*}
			Now, assume that \eqref{eq: lema pesos cond 5} holds for $n\in\N$. Then, 
			\begin{equation*}
				\begin{split}
					\widehat{V_{\mu_+,n+1}}(r) 
\gtrsim 
\int_{r}^{\frac{1+r}{2}}\mug(s)(1-s)^{n+1}dr 
\gtrsim \mug(r)(1-r)^{n+2}, \quad \frac{1}{2}\leq r <1.
				\end{split}
			\end{equation*}
This finishes the proof.
		\end{proof}

		\section{Fractional derivative characterization of Hardy space}\label{s3}
We begin this section by proving the following inequality.

		\begin{theorem}\label{th: norma equivalente <}
	Let  $\mu$ be a radial weight and $0<p<\infty$. Then,
			\begin{equation}\label{eq: th1 enunciado}
				\nm{f}_{H^p}^p \lesssim \int_\T \left [ \int_{\Gamma(\xi)} \abs{D^\mu(f)(z)}^2 \left (\frac{\mug(z)}{1-\abs{z}}\right )^2 dA(z) \right ]^{\frac{p}{2}}\abs{d\xi},\quad f\in \H(\D).
			\end{equation}
		\end{theorem}
We will use the following three lemmas in the proof of Theorem~\ref{th: norma equivalente <}.

\begin{lemma}\label{le:a}
Let $0<p<\infty$, $n\in\N$ and $a>0$. Then, 
$$\nm{f}_{H^p}^p \asymp \sum_{j=0}^{n-1}\abs{f^{(j)}(0)}^p + \int_{\T} \left (\int_{\Gamma(\z)} |z|^{2a}\abs{f^{(n)}(z)}^2(1-\abs{z})^{2n-2}dA(z) \right )^{\frac{p}{2}} \abs{d\z}, \quad  f\in\H(\D).$$
\end{lemma}
\begin{proof}
Since $M_p(r,g)$  is an increasing function of $r$ for any $g\in \H(\D)$, for any radial weight $\omega$
\begin{equation}\label{eq:j4}
\|g\|_{A^p_\om}\asymp \||z|^a g\|^p_{L^p_\om}, \quad  g\in\H(\D).
\end{equation}  
On the other hand, by 
\cite[Lemma~4]{arsenovic} 
there is $\lambda(p)=\lambda>0$ such that
\begin{equation*}\begin{split}
 & \int_{\T} \left (\int_{\Gamma(\z)} |z|^{2a}\abs{f^{(n)}(z)}^2(1-\abs{z})^{2n-2}dA(z) \right )^{\frac{p}{2}} \abs{d\z}
\\ & \asymp \int_{\T} \left (\int_{\D} \left(\frac{1-|z|}{|1-\overline{\z}z|} \right)^{\lambda} |z|^{2a}\abs{f^{(n)}(z)}^2(1-\abs{z})^{2n-2}\,dA(z) \right )^{\frac{p}{2}} \abs{d\z}, \quad  f\in\H(\D).
\end{split}\end{equation*}
So, by \eqref{eq:j4} with $g_\z(z)=\frac{f^{(n)}(z)}{(1-\overline{\z}z)^\lambda}$ and the radial weight $(1-\abs{z})^{2n-2+\lambda}$, and another application of  \cite[Lemma~4]{arsenovic} it follows that
\begin{equation*}\begin{split}
 & \int_{\T} \left (\int_{\Gamma(\z)} |z|^{2a}\abs{f^{(n)}(z)}^2(1-\abs{z})^{2n-2}dA(z) \right )^{\frac{p}{2}} \abs{d\z}
\\ & \asymp \int_{\T} \left (\int_{\D} \left(\frac{1-|z|}{|1-\overline{\z}z|} \right)^{\lambda} |z|^{2a}\abs{f^{(n)}(z)}^2(1-\abs{z})^{2n-2}\,dA(z) \right )^{\frac{p}{2}} \abs{d\z}
\\ & \asymp \int_{\T} \left (\int_{\D} \left(\frac{1-|z|}{|1-\overline{\z}z|} \right)^{\lambda} \abs{f^{(n)}(z)}^2(1-\abs{z})^{2n-2}\,dA(z) \right )^{\frac{p}{2}} \abs{d\z}
\\ & \asymp \int_{\T} \left (\int_{\Gamma(\z)} \abs{f^{(n)}(z)}^2(1-\abs{z})^{2n-2}dA(z) \right )^{\frac{p}{2}} \abs{d\z}, \quad  f\in\H(\D).
\end{split}\end{equation*}
This together with \eqref{eq: norma calderon} finishes the proof.
\end{proof}

Let us recall that, for any $0<p,q<\infty$ and $\nu$ a radial weight, the tent space $T_p^q(\nu)$ consists of complex-valued measurable functions $f$ on $\D$ such that 
	$$
	\|f\|_{T_{p}^{q}(\nu)}
	=  \left(\int_{\T}\left(\int_{\Gamma(\xi)}|f(z)|^p \nu(z)\frac{dA(z)}{1-|z|}\right)^{\frac{q}{p}}\,|d\xi|\right)^{\frac1q}<\infty.
$$
The next result follows from \cite[Lemma 3.1]{AguiMasPelRat}.
\begin{letterlemma}\label{le:b}
Let  $\om$ be a radial weight, $0<p<\infty$ and $n\in\N\cup\{0\}$. Then, 
 for any $0<r<1$  there is $C=C(n,p,\om,r)$ such that
\begin{equation}\label{eq:UC:ATp}
\sup_{|z|\le r}|f^{(n)}(z)|
\le C\|f\|_{AT^p_2(\om)},
\qquad f\in\H(\D).
\end{equation}
As a consequence, the convergence in $AT^p_2(\om)$ implies the uniform convergence on compacts subsets, and so $AT^p_2(\om)$ is a complete space.
\end{letterlemma}

Now, let us prove  following inequality which is a direct byproduct of  \cite[Lemma 2.5]{OrtegaFabrega}. 

\begin{letterlemma}\label{le:OAcase1}
Let $s>-1$, $r, t\ge 0$ such that $r+t-s> 2$. If $r-s<2$ and $t-s<2$, then  
$$\int_{\D}\frac{(1-|\z|)^s}{|1-\overline{\z}a|^r |1-\overline{\z}\xi|^t}\,dA(\z)\lesssim \frac{1}{ |1-\overline{a}\xi|^{t+r-s-2}}, \quad a\in \D, \xi\in \T.$$
\end{letterlemma}
\begin{proof}
By \cite[Lemma 2.5]{OrtegaFabrega} the above inequality is true whenever $a,\xi\in \D$. Therefore,
\begin{align*}
\int_{\D}\frac{(1-|\z|)^s}{|1-\overline{\z}a|^r |1-\overline{\z}\xi|^t}\,dA(\z) & = \int_{\D} \liminf_{\r\to 1^-}\frac{(1-|\z|)^s}{|1-\overline{\z}a|^r |1-\overline{\z}\r\xi|^t}\,dA(\z) 
\\ & \le \liminf_{\r\to 1^-}  \int_{\D} \frac{(1-|\z|)^s}{|1-\overline{\z}a|^r |1-\overline{\z}\textcolor{red}{\r}\xi|^t}\,dA(\z) 
\\ & \lesssim \liminf_{\r\to 1^-}\frac{1}{ |1-\overline{a}\r\xi|^{t+r-s-2}}=\frac{1}{ |1-\overline{a}\xi|^{t+r-s-2}},   \quad a\in \D, \xi\in \T.
\end{align*}
This finishes the proof.
\end{proof}

Let $\H(\overline{\D})$ the class of functions $f$  satisfying that there is  $\rho=\rho(f)>1$ such that $f$ is analytic in $D(0, \rho)$ and let be
 $dA_\alpha(z)=(\alpha+1)(1-|z|^2)^\alpha$, $\alpha>-1$.

		\begin{Prf}{\em{Theorem~\ref{th: norma equivalente <}.}}
Firstly we will assume that  $ f \in\H(\overline{\D})$ .  So, by Corollary~ \ref{prop: expresión derivada fraccionaria}
			$$
			D^{\mu}(f)(z) = \sum_{j=0}^{n-1} \frac{\widehat{f}(j)}{\mu_{2j+1}}z^j + 4^n z^n R^{W^n,V_{\mu^+,n}}(f^{(n)})(z),\quad z\in\D.
			$$
		Then, 	if we write $g=D^\mu(f)$ 
			$$
			R^{V_{\mu^+,n},W^n}(g)(z) = \sum_{j=0}^{n-1}\frac{(V_{\mu^+,n})_{2j+1}}{(W^n)_{2j+1}\mu_{2j+1}}\widehat{f}(j)z^j + 4^n z^n f^{(n)}(z), \quad z\in \D.
			$$
Next, bearing in mind Lemma~\ref{le:a} and Lemma~\ref{le:b}
\begin{equation*}\begin{split}
& \nm{f}_{H^p}^p  \asymp \sum_{j=0}^{n-1}\abs{f^{(j)}(0)}^p + \int_{\T} \left (\int_{\Gamma(\z)} |z|^{2n}\abs{f^{(n)}(z)}^2 \,dA_{2n-2}(z) \right )^{\frac{p}{2}} \abs{d\z}
\\ & \asymp \sum_{j=0}^{n-1}\abs{f^{(j)}(0)}^p + \int_{\T} \left (\int_{\Gamma(\z)} \left| R^{V_{\mu^+,n},W^n}(g)(z) - \sum_{j=0}^{n-1}\frac{(V_{\mu^+,n})_{2j+1}}{(W^n)_{2j+1}\mu_{2j+1}}\widehat{f}(j)z^j \right| ^2 
\,dA_{2n-2}(z) \right )^{\frac{p}{2}} \abs{d\z}
\\ & \lesssim \sum_{j=0}^{n-1}\abs{f^{(j)}(0)}^p + \int_{\T} \left (\int_{\Gamma(\z)} \left| R^{V_{\mu^+,n},W^n}(g)(z)\right| ^2\,dA_{2n-2}(z)\right )^{\frac{p}{2}} \abs{d\z}
\\ & \lesssim  \int_{\T} \left (\int_{\Gamma(\z)} \left| R^{V_{\mu^+,n},W^n}(g)(z)\right| ^2 \,dA_{2n-2}(z)\right )^{\frac{p}{2}} \abs{d\z}.
\end{split}\end{equation*}

			Then, we just need to prove 
\begin{equation}\label{eq:j6}
			\int_\T \left [ \int_{\Gamma(\xi)}\abs{R^{V_{\mu^+,n},W^n}(g)(z)}^2\,dA_{2n-2}(z) \right ]^{\frac{p}{2}}\abs{d\xi} \lesssim \int_\T \left [ \int_{\Gamma(\xi)}\abs{g(z)}^2\frac{\mug(z)^2}{(1-\abs{z})^2}dA(z) \right ]^{\frac{p}{2}}\abs{d\xi}.
		\end{equation}

Since $f\in \H(\overline{\D})$, then 	$g\in \H(\overline{\D})$ and therefore by Theorem~\ref{th:perala} and Lemma~\ref{lema pesos tec}(iii)
			\begin{equation*}
				\begin{split}
					\abs{R^{V_{\mu^+,n},W^n}(g)(z)}^2 
					&\lesssim \left (\int_\D \abs{g(\z)}\abs{B^{W^n}_z(\z)}\mug(\z)(1-\abs{\z})^{n}dA(\z) \right )^2,\quad z\in\D.
				\end{split}
			\end{equation*}
			Choose $n,N\in\N$ such that $2n+\frac{3}{2}>N>\frac{1}{2}\max\left\{1,\frac{2}{p}\right\}$ and  take 
\begin{equation}\label{eq:parameters}
 \ep\in \left(-1, 2n+2- \max\left\{1,\frac{2}{p}\right\}\right)\bigcap \left(2n+2-2N, 4n+2-2N \right).
\end{equation}
			Then, by Hölder's inequality and \cite[Lemma 5]{PelRatBMO} 
			\begin{equation}\label{eq: th1 desig hölder}
				\begin{split}
					\abs{R^{V_{\mu^+,n},W^n}(g)(z)}^2 & \lesssim 
\left (\int_\D \frac{\abs{g(\z)}^2}{\abs{1-\overline{z}\z}^{2N}}\mug(\z)^2 dA_{2n-\e}(\z) \right) \left (\int_{\D}\abs{(1-\overline{z}\z)^NB^{W^n}_z(\z)}^2 dA_\e(\z)\right ) \\
					&\lesssim \left (\int_\D \frac{\abs{g(\z)}^2}{\abs{1-\overline{z}\z}^{2N}}\mug(\z)^2 dA_{2n-\e}(\z) \right)\left (1+ \int_0^{\abs{z}}\frac{dr}{\widehat{W^{n}}(r)^2(1-r)^{1-2N-\e}} \right ),\quad z\in \D.
				\end{split}
			\end{equation}
			Now, bearing in mind that $4n+2-2N-\e$ is positive and Lemma~\ref{lema pesos tec} (i) it follows that
			\begin{equation}\label{eq: th1 desig Wn}
			1+ \int_0^{\abs{z}}\frac{dr}{\widehat{W^{n}}(r)^2(1-r)^{1-2N-\e}} \lesssim \frac{1}{(1-\abs{z})^{4n+2-2N-\e}},\quad z\in \D.
			\end{equation}
			Then by joining \eqref{eq: th1 desig hölder} and \eqref{eq: th1 desig Wn} 
			\begin{equation}\label{eq:A2}
			\abs{R^{V_{\mu^+,n},W^n}(g)(z)}^2 \lesssim \left (\int_\D \frac{\abs{g(\z)}^2}{\abs{1-\overline{z}\z}^{2N}}\mug(\z)^2dA_{2n-\e}(\z) \right )\frac{1}{(1-\abs{z})^{4n+2-2N-\e}},\quad z\in\D.
			\end{equation}
		 	Now take $\lambda>2n+2-\e$.  Then by  \cite[Lemma 4]{arsenovic}, Fubini's theorem,  Lemma~\ref{le:OAcase1} and \eqref{eq:parameters} we obtain that
		 	\begin{equation*}
		 		\begin{split}
		 			&\int_\T \left [ \int_{\Gamma(\xi)}\abs{R^{V_{\mu^+,n},W^n}(g)(z)}^2(1-\abs{z})^{2n-2}dA(z) \right ]^{\frac{p}{2}}\abs{d\xi} \\
		 			&\lesssim\int_\T \left [ \int_{\Gamma(\xi)}\left (\int_\D \frac{\abs{g(\z)}^2}{\abs{1-\overline{z}\z}^{2N}}\mug(\z)^2dA_{2n-\e}(\z) \right )\frac{1}{(1-\abs{z})^{2n+4-2N-\e}}dA(z) \right ]^{\frac{p}{2}}\abs{d\xi} \\
		 			&\lesssim 
\int_\T \left [ \int_\D \left (\frac{1-\abs{z}}{\abs{1-\overline{z}\xi}} \right)^\lambda \left (\int_\D \frac{\abs{g(\z)}^2}{\abs{1-\overline{z}\z}^{2N}}\mug(\z)^2dA_{2n-\e}(\z) \right )\frac{1}{(1-\abs{z})^{2n+4-2N-\e}}dA(z) \right ]^{\frac{p}{2}}\abs{d\xi} \\
		 			&= 
\int_\T \left [ \int_\D \abs{g(\z)}^2 \mug(\z)^2 (1-\abs{\z})^{2n-\e} \left (\int_\D \frac{(1-\abs{z})^{\lambda-2n-4+2N+\e}}{\abs{1-\overline{z}\xi}^\lambda\abs{1-\overline{z}\z}^{2N}} \, dA(z)\right )dA(\z)\right ]^{\frac{p}{2}}\abs{d\xi} \\
		 			&\lesssim 
\int_\T \left [\int_\D \left (\frac{1-\abs{\z}}{\abs{1-\overline{\z}\xi}} \right)^{2n+2-\e} \abs{g(\z)}^2 \frac{\mug(\z)^2}{(1-\abs{\z})^{2}}   dA(\z)\right ]^{\frac{p}{2}}\abs{d\xi}\\ 
		 			&\lesssim \int_\T \left [\int_{\Gamma(\xi)} \abs{g(\z)}^2 \frac{\mug(\z)^2}{(1-\abs{\z})^{2}}   dA(\z)\right ]^{\frac{p}{2}}\abs{d\xi},
		 		\end{split}
		 	\end{equation*}
where in the last inequality we have used that $\max\{1,\frac{2}{p}\}<2n+2-\e$. This proves \eqref{eq:j6} and finishes the proof for $f\in \H(\overline{\D})$.
		 	
		 	Next we will prove \eqref{eq: th1 enunciado} for any analytic function in $\D$.
Without loss of generality we may  assume that $\om(z)=\frac{\mug(z)^2}{1-\abs{z}}$ is a weight, otherwise the right-hand side equals infinity for any $f\in\H(\D)$, $f\neq 0$. 
(In fact, this is clear for $p=2$. And consequently, by H\"older's inequality it is also true  for $p>2$. For $0<p<2$,  bearing in mind \cite[Proposition 1.5]{AguiMasPelRat} and Minkowski's inequality we get
$$AT^p_2(\om)\subset AL^p_2(\om)\subset A^{p,2}_\om=\{f\in \H(\D): \int_0^1 M_p^2(r,f) \om(r)\,dr<\infty\}, $$
and therefore the assertion follows).

Consequently, let $f\in\H(\D)$ such that the right-hand side is finite , that is $D^\mu(f)\in AT^{p}_{2}(\om)$. For $f\in\H(\D)$ and $0\leq r <1$ we denote by $f_r(z)=f(rz)$, $z\in\D$. Then bearing in mind that $f_r\in \H( \overline{\D})$,  and \cite[Theorem 1.2 (i)]{AguiMasPelRat} 
		 	\begin{equation*}
		 		\begin{split}
		 			\nm{f}_{H^p}^p &= \lim_{r\to1^-} M_p^p(r,f) = \lim_{r\to1^-} \nm{f_r}^p_{H^p} \lesssim \lim_{r\to1^-}\nm{D^\mu(f_r)}_{AT^p_2(\om)}^p = \lim_{r\to1^-} \nm{(D^\mu(f))_r}_{AT^p_2(\om)}^p \\
		 			&\lesssim \nm{D^\mu(f)}_{AT^p_2(\om)} =  \int_\T \left [ \int_{\Gamma(\xi)} \abs{D^\mu(f)(z)}^2 \left (\frac{\mug(z)}{1-\abs{z}}\right )^2 dA(z) \right ]^{\frac{p}{2}}\abs{d\xi},
		 		\end{split}
		 	\end{equation*}
		 	concluding the proof.
		\end{Prf}

	\begin{theorem}\label{th: norma equivalente >}
		Let $\mu$ be a radial weight and  $0<p<\infty$. Then, the inequality
		\begin{equation}\label{eq: th2 enunciado}
		 \int_\T \left [ \int_{\Gamma(\xi)} \abs{D^\mu(f)(z)}^2 \left (\frac{\mug(z)}{1-\abs{z}}\right )^2 dA(z) \right ]^{\frac{p}{2}}\abs{d\xi}\lesssim \nm{f}_{H^p}^p,\quad f\in \H(\D),
		\end{equation}
	holds	if and only if $\mu\in \DDD$.
	\end{theorem}
	\begin{proof}
	Assume  that $\mu\in\DDD$. Firstly we will prove \eqref{eq: th2 enunciado} for  $f\in\H(\overline{\D})$. Then,  for each  $n\in\N$, by Corollary~\ref{prop: expresión derivada fraccionaria}  and Lemma~\ref{lema pesos tec} (i) 
		\begin{equation}
			\begin{split} \label{eq: th2 expresion desig fracc}
				\abs{D^\mu(f)(z)}^2 &\leq 4n^2\sum_{j=0}^{n-1} \frac{\abs{f^{(j)}(0)}^2}{(j!)^2\mu_{2j+1}^2} + 4^{2n+1}\abs{z}^{2n} \left (\int_\D \abs{f^{(n)}(\z)}\abs{B^{V_{\mu^+,n}}_z(\z)}W^n(\z)dA(\z) \right )^2 
\\ & \lesssim \sum_{j=0}^{n-1}\abs{f^{(j)}(0)}^2 + \left (\int_\D \abs{f^{(n)}(\z)} \abs{B^{V_{\mu^+,n}}_z(\z)} (1-\abs{\z})^{2n}dA(\z) \right )^2,
\quad z\in\D.
			\end{split}
		\end{equation}
Take $n\in\N\setminus\{1\}$ such that  $n>\frac{1}{p}-\frac{1}{2}>\frac{1}{p}-\frac{3}{2}$ and $N\in \N$ such that $$\frac{1}{p}<n+\frac{1}{2}<N<n+\frac{3}{2}+\frac{1}{2}\min\{2\b,1-\eta\},$$ 
where $\b>0$ comes from Lemma~\ref{lemma: caract dcheck} (ii) and  $\eta<1$ from Lemma~\ref{lemma: caract dcheck} (iii) with $\gamma=2$.
Next, choose $\e$ such that
\begin{equation}\label{eq:parameters2}
\e\in \left(\max\left\{\frac{2}{p}-2,0\right\}, 2n+1 \right)\bigcap \left(2N -2 -\min\{2\b,1-\eta\}, 2N-2 \right).
\end{equation}
		Since $V_{\mu^+,n}\in\DD$, by \cite[Lemma 5]{PelRatBMO} and Hölder's inequality
		\begin{equation}\label{eq: th2 eq tras holder}
			\begin{split}
				&\left (\int_\D \abs{f^{(n)}(\z)} \abs{B^{V_{\mu^+,n}}_z(\z)} (1-\abs{\z})^{2n}dA(\z) \right )^2
\\
				&\lesssim \left (\int_\D \frac{\abs{f^{(n)}(\z)}^2}{\abs{1-\overline{z}\z}^{2N}}dA_{2n+\e}(\z) \right ) \left ( \int_\D \abs{(1-\overline{z}\z)^N B^{V_{\mu^+,n}}_z(\z)}^2(1-\abs{\z})^{2n-\e}dA(\z)\right ) \\
				&\lesssim \left (\int_\D \frac{\abs{f^{(n)}(\z)}^2}{\abs{1-\overline{z}\z}^{2N}}dA_{2n+\e}(\z) \right ) \left (1+\int_0^{\abs{z}} \frac{dr}{\widehat{V_{\mu^+,n}}(r)^2 (1-r)^{1-2n-2N+\e}} \right )\quad z\in\D.
			\end{split}
		\end{equation} 
		Let us prove that 
		\begin{equation}\label{eq: th2 desig Vgorro}
			1+\int_0^{\abs{z}} \frac{dr}{\widehat{V_{\mu^+,n}}(r)^2 (1-r)^{1-2n-2N+\e}} \lesssim \frac{1}{\mug(z)^2(1-\abs{z})^{2-2N+\e}},\quad z\in \D.
		\end{equation}
Bearing in mind Lemma~\ref{lemma: caract dcheck} (ii) and \eqref{eq:parameters2} 
		(in particular $2+\e-2N+2\b>0$)  it is clear that \eqref{eq: th2 desig Vgorro} holds if $\abs{z}\leq \frac{1}{2}$.
		Next, if $\abs{z}\geq \frac{1}{2}$ 
by Lemma~\ref{lema pesos tec} (v), \eqref{eq:parameters2}   and Lemma~\ref{lemma: caract dcheck} (iii)
		\begin{equation*}
			\begin{split}
				&\int_{\frac{1}{2}}^{\abs{z}} \frac{dr}{\widehat{V_{\mu^+,n}}(r)^2 (1-r)^{1-2n-2N+\e}} \lesssim \int_{\frac{1}{2}}^{\abs{z}} \frac{dr}{\mug(r)^2(1-r)^{3-2N+\e}} \\
				&\leq  \int_{0}^{\abs{z}} \frac{dr}{\mug(r)^2(1-r)^{3-2N+\e}} \leq \frac{1}{(1-\abs{z})^{3-\eta-2N+\e}}\int_{0}^{\abs{z}} \frac{dr}{\mug(r)^2(1-r)^\eta} \\
				&\lesssim \frac{1}{\mug(z)^2(1-\abs{z})^{2-2N+\e}}, 
			\end{split}
		\end{equation*}
which implies that \eqref{eq: th2 desig Vgorro} holds.
		Then by joining \eqref{eq: th2 expresion desig fracc}, \eqref{eq: th2 eq tras holder} and \eqref{eq: th2 desig Vgorro} we obtain that
		\begin{equation}\label{eq:A3}
		\abs{D^\mu(f)(z)}^2 \lesssim \sum_{j=0}^{n-1}\abs{f^{(j)}(0)}^2 + \left (\int_\D \frac{\abs{f^{(n)}(\z)}^2}{\abs{1-\overline{z}\z}^{2N}}dA_{2n+\e}(\z) \right )\frac{1}{\mug(z)^2(1-\abs{z})^{2-2N+\e}},\quad z\in\D.
		\end{equation}
Now, by Lemma~\ref{lemma: caract dcheck} (ii) 
\begin{equation*}
			\begin{split}
				\int_\T \left [ \int_{\Gamma(\xi)} \left (\frac{\mug(z)}{1-\abs{z}}\right )^2 dA(z) \right ]^{\frac{p}{2}}\abs{d\xi}
\asymp \int_0^1 \frac{(\mug(s))^2}{1-s}\,ds\lesssim \int_0^1 \frac{1}{(1-s)^{1-2\beta}}\,ds<\infty.
\end{split}\end{equation*}
		Therefore
		\begin{equation*}
			\begin{split}
				&\int_\T \left [ \int_{\Gamma(\xi)} \abs{D^\mu(f)(z)}^2 \left (\frac{\mug(z)}{1-\abs{z}}\right )^2 dA(z) \right ]^{\frac{p}{2}}\abs{d\xi} \\ &\lesssim \left ( \sum_{j=0}^{n-1}\abs{f^{(j)}(0)}^p\right )\int_\T \left [ \int_{\Gamma(\xi)} \left (\frac{\mug(z)}{1-\abs{z}}\right )^2 dA(z) \right ]^{\frac{p}{2}}\abs{d\xi} \\
				&+ \int_\T \left [ \int_{\Gamma(\xi)} \left (\int_\D \frac{\abs{f^{(n)}(\z)}^2}{\abs{1-\overline{z}\z}^{2N}}dA_{2n+\e}(\z) \right )\frac{1}{\mug(z)^2(1-\abs{z})^{2-2N+\e}} \left (\frac{\mug(z)}{1-\abs{z}}\right )^2 dA(z) \right ]^{\frac{p}{2}}\abs{d\xi} \\
				&\lesssim \sum_{j=0}^{n-1}\abs{f^{(j)}(0)}^p + \int_\T \left [ \int_{\Gamma(\xi)} \left (\int_\D \frac{\abs{f^{(n)}(\z)}^2}{\abs{1-\overline{z}\z}^{2N}}dA_{2n+\e}(\z) \right )\frac{1}{(1-\abs{z})^{4-2N+\e}} dA(z) \right ]^{\frac{p}{2}}\abs{d\xi}.
			\end{split}
		\end{equation*}
		So we just need to show that the second term in the sum can be dominated by the second term in \eqref{eq: norma calderon}. \\
		Take $2+\e<\lambda$.  Then by  \cite[Lemma 4]{arsenovic},  Fubini’s theorem,
		  Lemma~\ref{le:OAcase1}  and \eqref{eq:parameters2} we obtain that
		\begin{equation*}
			\begin{split}
			&\int_\T \left [ \int_{\Gamma(\xi)} \left (\int_\D \frac{\abs{f^{(n)}(\z)}^2}{\abs{1-\overline{z}\z}^{2N}}dA_{2n+\e}(\z) \right )\frac{1}{(1-\abs{z})^{4-2N+\e}} dA(z) \right ]^{\frac{p}{2}}\abs{d\xi} \\
			&\lesssim \int_\T \left [\int_\D \left (\frac{(1-\abs{z})}{\abs{1-\overline{z}\xi}}\right )^\lambda\left (\int_\D \frac{\abs{f^{(n)}(\z)}^2}{\abs{1-\overline{z}\z}^{2N}}dA_{2n+\e}(\z) \right )\frac{1}{(1-\abs{z})^{4-2N+\e}} dA(z) \right ]^{\frac{p}{2}}\abs{d\xi} \\
			&= \int_\T \left [\int_\D \abs{f^{(n)}(\z)}^2(1-\abs{\z})^{2n+\e} \left (\int_\D \frac{(1-\abs{z})^{\lambda+2N-4-\e}}{\abs{1-\overline{z}\xi}^\lambda\abs{1-\overline{z}\z}^{2N}}dA(z) \right) dA(\z) \right  ]^{\frac{p}{2}}\abs{d\xi} \\
			&\lesssim 
\int_\T \left [\int_\D \left (\frac{1-\abs{\z}}{\abs{1-\overline{\z\xi}}}\right )^{2+\e} \abs{f^{(n)}(\z)}^2(1-\abs{\z})^{2n-2}  dA(\z) \right  ]^{\frac{p}{2}}\abs{d\xi} \\
			&\lesssim \int_\T \left [\int_{\Gamma(\xi)}  \abs{f^{(n)}(\z)}^2(1-\abs{\z})^{2n-2}  dA(\z) \right  ]^{\frac{p}{2}}\abs{d\xi},
			\end{split}
		\end{equation*}
		concluding the proof for $f\in\H(\overline{\D})$. 

		Now, if  $f\in H^p$ by  \cite[Theorem 1.2 (ii)]{AguiMasPelRat}
		\begin{equation*}
			\begin{split}
		  &\int_\T \left [ \int_{\Gamma(\xi)} \abs{D^\mu(f)(z)}^2 \left (\frac{\mug(z)}{1-\abs{z}}\right )^2 dA(z) \right ]^{\frac{p}{2}}\abs{d\xi}
		  = \nm{D^\mu(f)}_{AT^p_2(\om)}^p\\ &\lesssim \limsup_{r\to1^-} \nm{(D^\mu(f))_r}_{AT^p_2(\om)}^p 
		  =\limsup_{r\to1^-}\nm{D^\mu(f_r)}_{AT^p_2(\om)}^p \lesssim \limsup_{r\to1^-}\nm{f_r}_{H^p}^p = \nm{f}_{H^p}^p.
			\end{split}
		\end{equation*}
This proves \eqref{eq: th2 enunciado}.

	Reciprocally assume  that \eqref{eq: th2 enunciado} holds and let us prove that $\mu\in\DDD$. 
	Testing  \eqref{eq: th2 enunciado} with $f(z)=z^n$,  $n\in\N\cup\{0\}$, we obtain the inequality
	\begin{equation*}
		\begin{split}
	\left [\int_0^1 \frac{\mug(r)^2}{1-r} r^{2n+1}  dr \right]^{\frac{p}{2}} & \asymp \int_\T\left[\int_{\Gamma(\xi)}\abs{z}^{2n}\frac{\mug(z)^2}{(1-\abs{z})^2}dA(z) \right ]^{\frac{p}{2}}\abs{d\xi}
\\ &  \lesssim \mu_{2n+1}^p,\quad n\in\N\cup\{0\},
\end{split}
	\end{equation*}
	that is
	\begin{equation}\label{eq: th2 necesidad discreto}
		\int_0^1 \frac{\mug(r)^2}{1-r}r^{2n+1}dr \lesssim \mu_{2n+1}^2,\quad n\in \N\cup\{0\}.
	\end{equation}
If $x\geq1$ take $n\in \N$ such that $2n+1\leq x < 2n+3$.  Then,
	$$
	\int_0^1 \frac{\mug(r)^2}{1-r}r^xdr \leq \int_0^1 \frac{\mug(r)^2}{1-r}r^{2n+1}dr \lesssim \mu_{2n+1}^2 \lesssim \mu_{2n+3}^2\leq \mu_x^2,
	$$
where in the second to last inequality we have used that  there exists a constant $C(\mu)>0$ such that $\mu_x\leq C(\mu)\mu_{x+2}$ for every $x>0$. 
	Therefore,
	\begin{equation}\label{eq: th2 necesidad continuo}
		\int_0^1 \frac{\mug(r)^2}{1-r}r^xdr \lesssim \mu_x^2,\quad x\geq 1.
	\end{equation}
	So,  an integration by parts,  Hölder's inequality and \eqref{eq: th2 necesidad continuo} imply that
	\begin{equation*}
		\begin{split}
			\mu_{\frac{3x}{4}+1}^2 &\asymp x^2 \left (\int_0^1 \mug(r)r^{\frac{3x}{4}}dr\right )^2 
			\leq x^2 \left (\int_0^1\frac{\mug(r)^2}{1-r}r^xdr \right ) \left (\int_{0}^1 (1-r)r^{\frac{x}{2}}dr \right ) 
\\ & \asymp \int_0^1 \frac{\mug(r)^2}{1-r}r^xdr \lesssim \mu_x^2, \quad x\geq1.
		\end{split}
	\end{equation*}
	Consequently,
	$
	\mu_x\lesssim \mu_{\frac{4x}{3}},\,x\geq,  \frac{7}{4}.
	$
	Next, three iterations of this inequality yield
	$$
	\mu_x \lesssim \mu_{\frac{4x}{3}} \lesssim \mu_{\frac{16x}{9}} \lesssim \mu_{\frac{64x}{27}} \leq \mu_{2x},\quad x\geq \frac{7}{4},
	$$
	so by Lemma~\ref{lemma: caract Dgorro} (iii) $\mu\in\DD$. \\
	Let us prove  that $\mu\in \Dd$. If $K>1$, 
	\begin{equation*}\begin{split}
			\int_0^1 r^{x} \frac{\widehat{\mu}(r)^2}{1-r}\,dr  \ge \int_{1-\frac{1}{x}}^{1-\frac{1}{Kx}} r^{x} \frac{\widehat{\mu}(r)^2}{1-r}\,dr
			 \ge C_1  \widehat{\mu}\left(1-\frac{1}{Kx}\right)^2\log K, \quad x\ge 2,
	\end{split}\end{equation*}
	where $C_1=\inf_{x\ge 2}\left( 1-\frac{1}{x}\right)^x>0$.
	The above inequality, together with \eqref{eq: th2 necesidad continuo}, the fact that $\mu\in \DD$ and Lemma~\ref{lemma: caract Dgorro}, implies that there is $C_2=C_2(\mu)>0$ such that
	\begin{equation*}\begin{split}
			C_1  \widehat{\mu}\left(1-\frac{1}{Kx}\right)^2\log K \le  \int_0^1 r^{x} \frac{\widehat{\mu}(r)^2}{1-r}\,dr 
			\le C_2 \widehat{\mu}\left(1-\frac{1}{x} \right) ^2, \quad x\ge 2.
	\end{split}\end{equation*}
	That is,
	$$\widehat{\mu}(r)\ge \left( \frac{C_1\log K}{C_2}\right)^{1/2}\widehat{\mu}\left(1 - \frac{1-r}{K}\right), \quad \frac{1}{2}\le r<1.$$
	So, choosing $\log K>\frac{C_2}{C_1}$, it follows that $\mu\in \Dd$, and consequently $\mu\in\DDD$.
	\end{proof}
	
 As a byproduct of Fubini's theorem, Theorem~\ref{th: norma equivalente <} and Theorem~\ref{th: norma equivalente >} we obtain the following 
 Littlewood-Paley type inequalities for the Hardy space $H^2$.  
	\begin{corollary}\label{coro: H2}
		Let $\mu$ be a radial weight, then
$$
		\nm{f}_{H^2}^2 \lesssim \int_\D \abs{D^\mu(f)(z)}^2 \frac{\mug(z)^2}{1-\abs{z}}\,dA(z),\quad f\in\H(\D).
		$$
Moreover, the following conditions are equivalent:
\begin{enumerate}
\item[\rm(i)] $\mu\in\DDD$;
\item[\rm(ii)] $$
		\nm{f}_{H^2}^2 \asymp \int_\D \abs{D^\mu(f)(z)}^2 \frac{\mug(z)^2}{1-\abs{z}}\,dA(z),\quad f\in\H(\D);
		$$
\item[\rm(iii)]
		$$
		 \int_\D \abs{D^\mu(f)(z)}^2 \frac{\mug(z)^2}{1-\abs{z}}\,dA(z)\lesssim \nm{f}_{H^2}^2,\quad f\in\H(\D).
		$$
\end{enumerate}
	\end{corollary}

\section{Fractional derivative characterization of BMOA and VMOA}\label{s4}

We will use some  well-known results. The first one follows from \cite[Lemmas 2.4 and 3.3]{BlascoJarchow} (see also \cite[Lemma~3.3 p. 231]{Garnett1981}).
\begin{letterlemma}\label{lemma: carleson nucleo}
	Let $\mu$ be a positive  Borel measure on $\D$. Then
\begin{enumerate}
\item[\rm(i)] For each
  $\lambda>0$, there exists positive constants $C_1=C_1(\lambda), C_2=C_2(\lambda)>0$ such that
	$$
	C_1\sup_{a\in\D} \frac{\mu(S(a))}{1-\abs{a}} \leq \sup_{a\in\D}\int_{\D} \frac{(1-\abs{a})^\lambda}{\abs{1-\overline{a}z}^{\lambda+1}}d\mu(z) \leq C_2 \sup_{a\in\D} \frac{\mu(S(a))}{1-\abs{a}};
	$$
\item[\rm(ii)] For each
  $\lambda>0$, $ \lim_{|a|\to 1^-} \int_{\D} \frac{(1-\abs{a})^\lambda}{\abs{1-\overline{a}z}^{\lambda+1}}d\mu(z)=0
	$ if and only if $$ \lim_{|a|\to 1^-} \frac{\mu(S(a))}{1-\abs{a}}=0.
	$$
\end{enumerate}
\end{letterlemma}

The second result is a particular case ($p=1$) of 
 \cite[Theorem 5]{AulaskariNowakZhao}.
\begin{lettertheorem}\label{th: derivada n-esima bmoa}
	Let $n\in\N$. 
\begin{enumerate}
\item[\rm(i)] 
 $f\in\BMOA$ if and only if $d\mu_{f,n}(z) = \abs{f^{(n)}(z)}^2(1-\abs{z})^{2n-1}dA(z)$ is a classical Carleson measure. Moreover,
	$$
	\nm{f}_{\BMOA}^2\asymp \sum_{k=0}^{n-1} \abs{f^{(k)}(0)}^2 + \sup_{a\in\D} \frac{\int_{S(a)}\abs{f^{(n)}(z)}^2(1-\abs{z})^{2n-1}dA(z)}{1-\abs{a}};
	$$
\item[\rm(ii)] 
 $f\in\VMOA$ if and only if $d\mu_{f,n}(z) = \abs{f^{(n)}(z)}^2(1-\abs{z})^{2n-1}dA(z)$ is a vanishing Carleson measure.
\end{enumerate}
\end{lettertheorem}

We are ready to prove the first main result of this section.

\begin{theorem}\label{th: caracterizacion BMOA <}
	Let $\mu$ be a radial weight. Then, 
	$$
	\nm{f}_{\BMOA} \lesssim \|f\|_{\BMOA_\mu},\quad f\in\H(\D).
	$$
Moreover,  $ \VMOA_\mu \subset \VMOA$.
\end{theorem}

\begin{proof}
	Let $f\in\H(\D)$ such that $d\mu_f(z)=\abs{D^\mu(f)(z)}^2\frac{\mug(z)^2}{1-\abs{z}}\,dA(z)$ is a classical Carleson measure.  
Then, by Lemma~\ref{lemma: carleson nucleo}  and Corollary~\ref{coro: H2} 
\begin{equation}\label{eq:A1}
\infty>\|f\|^2_{\BMOA_\mu}\gtrsim \int_\D \abs{D^\mu(f)(z)}^2 \frac{\mug(z)^2}{1-\abs{z}}\,dA(z)\gtrsim \|f\|^2_{H^2}.
\end{equation}
 So, by Corollary~ \ref{prop: expresión derivada fraccionaria}
$$
	4^n z^n f^{(n)}(z) = R^{V_{\mu^+,n},W^n}(g)(z) - \sum_{j=0}^{n-1}\frac{(V_{\mu^+,n})_{2j+1}}{(W^n)_{2j+1}\mu_{2j+1}}\widehat{f}(j)z^j, \quad n\in\N,
	$$
where $g=D^\mu(f)$. Therefore, by Lemma~\ref{lemma: carleson nucleo} and Theorem~\ref{th: derivada n-esima bmoa}, for each $n\in\N$ and $\lambda>0$

	\begin{equation}\label{eq: desig lambda 1}
		\begin{split}
			\nm{f}_{\BMOA}^2 
			&\asymp \sum_{k=0}^{n-1}\abs{f^{(k)}(0)}^2 + \sup_{a\in\D} \int_\D \abs{z^nf^{(n)}(z)}^2\frac{(1-\abs{z})^{2n-1}(1-\abs{a})^\lambda}{\abs{1-\overline{a}z}^{\lambda+1}}dA(z)
\\ &\lesssim   \|f\|^2_{\BMOA_\mu} + \sup_{a\in\D} \int_\D \abs{R^{V_{\mu^+,n},W^n}(g)(z)}^2\frac{(1-\abs{z})^{2n-1}(1-\abs{a})^\lambda}{\abs{1-\overline{a}z}^{\lambda+1}}dA(z) \\
			&+ \sup_{a\in\D} \int_\D \abs{\sum_{j=0}^{n-1}\frac{(V_{\mu^+,n})_{2j+1}}{(W^n)_{2j+1}\mu_{2j+1}}\widehat{f}(j)z^j}^2\frac{(1-\abs{z})^{2n-1}(1-\abs{a})^\lambda}{\abs{1-\overline{a}z}^{\lambda+1}}dA(z) \\
			&\lesssim \|f\|^2_{\BMOA_\mu} + \sup_{a\in\D} \int_\D \abs{R^{V_{\mu^+,n},W^n}(g)(z)}^2\frac{(1-\abs{z})^{2n-1}(1-\abs{a})^\lambda}{\abs{1-\overline{a}z}^{\lambda+1}}dA(z).
		\end{split}
	\end{equation}
	Then, by Lemma~\ref{lemma: carleson nucleo} and Theorem~\ref{th: derivada n-esima bmoa} we just need to prove that  
	$$
	\sup_{a\in\D} \int_\D \abs{R^{V_{\mu^+,n},W^n}(g)(z)}^2\frac{(1-\abs{z})^{2n-1}(1-\abs{a})^\lambda}{\abs{1-\overline{a}z}^{\lambda+1}}dA(z) \lesssim  \sup_{a\in\D}\int_{\D}\abs{g(z)}^2\frac{(1-\abs{a})^\lambda}{\abs{1-\overline{a}z}^{\lambda+1}}\frac{\mug(z)^2}{1-\abs{z}}dA(z),
	$$
 for some $n\in\N$ and $\lambda>0$.
	Take now $n,N\in\N$ such that $1<N<2n+\frac{3}{2}$ and 
\begin{equation}\label{eq:parametersBMOA1}
	\e\in (-1,2n+1)\cap(2n +2-2N,4n+2-2N).
	\end{equation}
Observe that by Lemma~\ref{lema pesos tec} and \eqref{eq:A1}
\begin{equation*}\begin{split}
\|g\| _{A^1_{V_{\mu^+,n}}} &\lesssim \int_{\D} |g(z)| \mug(z)(1-|z|)^n\,dA(z)
\\ &
\lesssim \left(\int_\D \abs{D^\mu(f)(z)}^2 \frac{\mug(z)^2}{1-\abs{z}}\,dA(z) \right)^{1/2}\left(\int_{\D} (1-\abs{z})^{2n+1}\,dA(z) \right)^{1/2}
\\ &\lesssim \left(\int_\D \abs{D^\mu(f)(z)}^2 \frac{\mug(z)^2}{1-\abs{z}}\,dA(z) \right)^{1/2}<\infty.
\end{split}\end{equation*}
So, we can apply Theorem~\ref{th:perala} and
	argue as in the proof of Theorem~\ref{th: norma equivalente <} to obtain \eqref{eq:A2}. That is, 
	$$
	\abs{R^{V_{\mu^+,n},W^n}(g)(z)}^2 \lesssim \left (\int_\D\frac{\abs{g(\z)}^2}{\abs{1-\overline{z}\z}^{2N}}\mug(\z)^2dA_{2n-\e} \right )\frac{1}{(1-\abs{z})^{4n+2-2N-\e}},\quad z\in\D.
	$$
	Then if $0<\lambda<2N-2n+\e-2$, by Fubini's theorem,  \cite[Lemma 2.5]{OrtegaFabrega}, \eqref{eq:parametersBMOA1} and Lemma~\ref{lemma: carleson nucleo}
	\begin{equation}\label{eq: desig lambda 2}
		\begin{split}
			&\int_\D \abs{R^{V_{\mu^+,n},W^n}(g)(z)}^2\frac{(1-\abs{z})^{2n-1}(1-\abs{a})^\lambda}{\abs{1-\overline{a}z}^{\lambda+1}}dA(z) \\
			&\lesssim \int_\D \frac{(1-\abs{z})^{2N+\e-2n-3}(1-\abs{a})^\lambda}{\abs{1-\overline{a}z}^{\lambda+1}}\left (\int_\D\frac{\abs{g(\z)}^2}{\abs{1-\overline{z}\z}^{2N}}\mug(\z)^2dA_{2n-\e}(\z) \right )dA(z) \\
			&= \int_\D \abs{g(\z)}^2 \mug(\z)^2(1-\abs{\z})^{2n-\e}(1-\abs{a})^\lambda \left (\int_\D\frac{(1-\abs{z})^{2N+\e-2n-3}}{\abs{1-\overline{a}z}^{\lambda+1}\abs{1-\overline{z}\z}^{2N}} dA(z)\right )dA(\z) \\
			&\lesssim 
\int_{\D}\abs{g(\z)}^2\frac{(1-\abs{a})^\lambda}{\abs{1-\overline{a}\z}^{\lambda+1}}\frac{\mug(\z)^2}{1-\abs{\z}}dA(\z) \lesssim \|f\|^2_{\BMOA_\mu}, \quad a\in\D.
		\end{split}
	\end{equation}
	This finishes the proof of the first part of the theorem.
\par Now assume that $f\in \VMOA_\mu$, and in particular  $f\in \BMOA_\mu$.
On the rest of the proof we keep the notation, the choice of the paremeters   and the estimates obtained in the first part of the proof.
 Therefore, 

\begin{equation}\begin{split}\label{eq:vmoa1}
& \int_\D \abs{f^{(n)}(z)}^2\frac{(1-\abs{z})^{2n-1}(1-\abs{a})^\lambda}{\abs{1-\overline{a}z}^{\lambda+1}}dA(z)
\\ & \lesssim  \int_\D \abs{R^{V_{\mu^+,n},W^n}(g)(z)}^2\frac{(1-\abs{z})^{2n-1}(1-\abs{a})^\lambda}{\abs{1-\overline{a}z}^{\lambda+1}}dA(z) 
\\ &			+  \int_\D \abs{\sum_{j=0}^{n-1}\frac{(V_{\mu^+,n})_{2j+1}}{(W^n)_{2j+1}\mu_{2j+1}}\widehat{f}(j)z^j}^2\frac{(1-\abs{z})^{2n-1}(1-\abs{a})^\lambda}{\abs{1-\overline{a}z}^{\lambda+1}}dA(z), \quad a\in \D. 
\end{split}\end{equation}
Moreover, on the one hand
\begin{equation}\begin{split}\label{eq:vmoa2}
& \lim_{|a|\to 1^-}\int_\D \abs{\sum_{j=0}^{n-1}\frac{(V_{\mu^+,n})_{2j+1}}{(W^n)_{2j+1}\mu_{2j+1}}\widehat{f}(j)z^j}^2\frac{(1-\abs{z})^{2n-1}(1-\abs{a})^\lambda}{\abs{1-\overline{a}z}^{\lambda+1}}dA(z)
\\ & \lesssim  \lim_{|a|\to 1^-} \|f\|^2_{\BMOA,\mu} \int_\D \frac{(1-\abs{z})^{2n-1}(1-\abs{a})^\lambda}{\abs{1-\overline{a}z}^{\lambda+1}}dA(z)=0.
\end{split}\end{equation}
On the other hand, bearing in mind Lemma~\ref{lemma: carleson nucleo}(ii)
\begin{equation}\begin{split}\label{eq:vmoa3}
& \lim_{|a|\to 1^-} \int_\D \abs{R^{V_{\mu^+,n},W^n}(g)(z)}^2\frac{(1-\abs{z})^{2n-1}(1-\abs{a})^\lambda}{\abs{1-\overline{a}z}^{\lambda+1}}dA(z) 
\\ & \lesssim  \lim_{|a|\to 1^-} \int_{\D}\abs{g(\z)}^2\frac{(1-\abs{a})^\lambda}{\abs{1-\overline{a}\z}^{\lambda+1}}\frac{\mug(\z)^2}{1-\abs{\z}}dA(\z) =0.
\end{split}\end{equation}
Consequently, joining \eqref{eq:vmoa1},  \eqref{eq:vmoa2} and \eqref{eq:vmoa3} it follows that
$$   \lim_{|a|\to 1^-} \int_\D \abs{f^{(n)}(z)}^2\frac{(1-\abs{z})^{2n-1}(1-\abs{a})^\lambda}{\abs{1-\overline{a}z}^{\lambda+1}}dA(z)=0.$$
So, by  Theorem~\ref{th: derivada n-esima bmoa}(ii), $f\in \VMOA$. This finishes the proof.
\end{proof}

\begin{theorem}\label{th: caracterizacion BMOA >}
	Let $\mu$ be a radial weight. Then, the following conditions are equivalent:
\begin{enumerate}
\item[\rm(i)] $\mu\in\DDD$;
\item[\rm(ii)] $\BMOA\subset  {\BMOA_\mu}$ and
	\begin{equation}\label{eq:BMOA >}
	\|f\|_{\BMOA_\mu}\lesssim \nm{f}_{\BMOA},\quad f\in\H(\D);
	\end{equation}
\item[\rm(iii)] $\VMOA\subset  {\VMOA_\mu}$ and
	\begin{equation}\label{eq:VMOA >}
	\|f\|_{\BMOA_\mu}\lesssim \nm{f}_{\BMOA},\quad f\in VMOA.
	\end{equation}
\end{enumerate}
\end{theorem}
\begin{proof}
	Assume that $\mu\in\DDD$ and $f\in\BMOA$. Then,  Corollary~\ref{prop: expresión derivada fraccionaria}  and the argument  in the proof of Theorem~\ref{th: norma equivalente >} imply that 
	$$
	\abs{D^\mu(f)(z)}^2 \lesssim \nm{f}_{\BMOA}^2 + \left (\int_\D \abs{f^{(n)}(\z)}\abs{B^{V_{\mu_+,n}}_z(\z)}(1-\abs{\z})^{2n}dA(\z) \right )^2,\quad n\in\N,  z\in\D.
	$$
	Take $n\in\N\setminus\{1\},N\in\N$  such that $$n+\frac{1}{2}<N<n+\frac{3}{2}+\frac{1}{2}\min\{2\b,1-\eta\},$$ where $\b>0$ comes from Lemma~\ref{lemma: caract dcheck} (ii) and  $\eta<1$ from Lemma~\ref{lemma: caract dcheck} (iii) with $\gamma=2$.
	Next, choose $\e$ such that
	\begin{equation}\label{eq:parametersBMOA2}
		\e\in \left(0, 2n+1 \right)\bigcap \left(2N -2 -\min\{2\b,1-\eta\}, 2N-2 \right).
	\end{equation}
	Then,    following  the proof of Theorem~\ref{th: norma equivalente >} we get \eqref{eq:A3}. Therefore, 
	\begin{equation}
	\abs{D^\mu(f)(z)}^2\lesssim \nm{f}_{\BMOA}^2 + \left (\int_\D \frac{\abs{f^{(n)}(\z)}^2}{\abs{1-\overline{z}\z}^{2N}}dA_{2n+\e}(\z) \right)\cdot\frac{1}{\mug(z)^2(1-\abs{z})^{2-2N+\e}},\quad z\in\D.
	\end{equation}
	 Then for any $a\in\D$ and $\lambda>0$ 
	\begin{equation*}
		\begin{split}
			&\int_\D\frac{(1-\abs{a})^\lambda}{\abs{1-\overline{a}z}^{\lambda+1}}\abs{D^\mu(f)(z)}^2\frac{\mug(z)^2}{1-\abs{z}}dA(z) 
\lesssim \nm{f}_{\BMOA}^2 \int_\D \frac{(1-\abs{a})^\lambda}{\abs{1-\overline{a}z}^{\lambda+1}}\frac{\mug(z)^2}{1-\abs{z}}dA(z) \\
			&+ \int_\D \frac{(1-\abs{a})^\lambda}{\abs{1-\overline{a}z}^{\lambda+1}(1-\abs{z})^{3-2N+\e}}\left (\int_\D \frac{\abs{f^{(n)}(\z)}^2}{\abs{1-\overline{z}\z}^{2N}}dA_{2n+\e}(\z) \right)dA(z) = I(f,a) + II(f,a).
		\end{split}
	\end{equation*}
	By Lemma~\ref{lemma: caract dcheck} (ii) 
	\begin{equation*}
			\int_\D \frac{\mug(z)^2}{1-\abs{z}}\frac{(1-\abs{a})^\lambda}{\abs{1-\overline{a}{z}}^{\lambda+1}}dA(z) \lesssim  \int_0^1 \frac{\mug(s)^2}{1-s}\frac{(1-\overline{a})^\lambda}{(1-\abs{a}s)^\lambda}ds \leq  \int_0^1 \frac{\mug(s)^2}{1-s}ds <\infty,
	\end{equation*}
	so $\sup_{a\in \D}I(f,a)\lesssim \nm{f}_{\BMOA}^2$. 
	Now take $0<\lambda<2N-2-\e$. Then by Fubini's Theorem, \cite[Lemma 2.5]{OrtegaFabrega}, \eqref{eq:parametersBMOA2}, Lemma~\ref{lemma: carleson nucleo} and Theorem~\ref{th: derivada n-esima bmoa}
	\begin{equation*}
		\begin{split}
			II(f,a) &= (1-\abs{a})^\lambda\int_\D \abs{f^{(n)}(\z)}^2(1-\abs{\z})^{2n+\e}\left (\int_\D \frac{(1-\abs{z})^{2N-3-\e}}{\abs{1-\overline{z}\z}^{2N}\abs{1-\overline{a}z}^{\lambda+1}}dA(z) \right)dA(\z) \\
			&\asymp 
 \int_\D \abs{f^{(n)}(\z)}^2 (1-\abs{\z})^{2n-1}\frac{(1-\abs{a})^{\lambda}}{\abs{1-\overline{a}\z}^{\lambda+1}}dA(\z) \lesssim \nm{f}_{\BMOA}^2.
		\end{split}
	\end{equation*}
	Then, by Lemma~\ref{lemma: carleson nucleo}(i), statement (ii) holds.

Now assume that  $\mu\in\DDD$ and $f\in\VMOA$. Then arguing as in the proof of (i)$\Rightarrow$(ii),  keeping the notation and the choice of the paremeters, it follows that
\begin{equation*}
		\begin{split}
			\int_\D\frac{(1-\abs{a})^\lambda}{\abs{1-\overline{a}z}^{\lambda+1}}\abs{D^\mu(f)(z)}^2\frac{\mug(z)^2}{1-\abs{z}}dA(z) \lesssim  I(f,a) + II(f,a), \quad a\in \D. 
\end{split}
	\end{equation*}
Now, since $\frac{\mug(z)^2}{1-\abs{z}}$ is radial weight,   $\frac{\mug(z)^2}{1-\abs{z}}\, dA(z)$ is a vanishing Carleson measure. So bearing in mind Lemma~\ref{lemma: carleson nucleo}(ii), 
$  \lim_{|a|\to 1^-}I(f,a)=0.$
Moreover,  by Lemma~\ref{th: derivada n-esima bmoa}(ii)
$$ \lim_{|a|\to 1^-}II(f,a)\lesssim  \lim_{|a|\to 1^-} \int_\D \abs{f^{(n)}(\z)}^2 (1-\abs{\z})^{2n-1}\frac{(1-\abs{a})^{\lambda}}{\abs{1-\overline{a}\z}^{\lambda+1}}dA(\z) = 0.$$
So, $ \lim_{|a|\to 1^-} \int_\D\frac{(1-\abs{a})^\lambda}{\abs{1-\overline{a}z}^{\lambda+1}}\abs{D^\mu(f)(z)}^2\frac{\mug(z)^2}{1-\abs{z}}dA(z)=0$, and  by Lemma~\ref{lemma: carleson nucleo}(ii), statement (iii) holds.

Reciprocally, assume  that \eqref{eq:BMOA >} holds and let us prove that $\mu\in\DDD$.
	By chosing  $f(z)=z^n$ in  \eqref{eq:BMOA >} we get
	 the inequality
	$$
	\int_\D \frac{\abs{z^{2n}}}{\mu_{2n+1}^2}\frac{\mug(z)^2}{1-\abs{z}}\,dA(z) \lesssim \nm{z^n}_{\BMOA}^2= 1,\quad n\in\N .
	$$
	Therefore, we obtain the inequality \eqref{eq: th2 necesidad discreto} which implies $\mu\in\DDD$ by the proof of Theorem~\ref{th: norma equivalente >}. Consequently (i)$\Leftrightarrow$(ii). 
Finally, since the monomials belong to $\VMOA$,  the same proof shows that (iii)$\Rightarrow$(i). This finishes the proof.
\end{proof}

\section{Proofs of Theorem~\ref{th:Vmubounded} and Theorem~\ref{th:Vmucompact}.}\label{s5}
Let $I$ be  the identity operator and for $0<p,q<\infty$ and a positive Borel measure $\nu$.  

The following result was proved in \cite[Theorem~9]{PRS}, see also \cite{Cohn}. 
\begin{lettertheorem}\label{th:cohn}
Let $0<p,q<\infty$ and $\nu$ be a positive Borel measure on $\D$. Then,
\begin{enumerate}
\item[\rm(i)]
 $I: H^p\to  T_q^p(\nu)$ is bounded 
if and only if $\nu$ is a classical Carleson measure. 
Moreover, $$\|I\|^q_{H^p\to T_q^p(\nu) }\asymp \sup_{a \in \D} \frac{\nu(S(a))}{1-|a|}.$$
\item[\rm(ii)]
 $I: H^p\to  T_q^p(\nu)$ is compact 
if and only if $\nu$ is a vanishing Carleson measure. 
\end{enumerate}
\end{lettertheorem}

\begin{Prf}{\em{Theorem~\ref{th:Vmubounded}}.}
By Theorem~\ref{th: caracterizacion BMOAmain} (ii)$\Leftrightarrow$(iii) and 	 $$\|g\|_{\BMOA,\mu} \asymp \nm{g}_{\BMOA},\quad g\in\H(\D).$$  
Next, by Theorem~\ref{th:main} $V_{\mu,g}$ is bounded on $H^p$ if and only if $I: H^p\to   T_2^p(\nu_g)$ is bounded where $d\nu_g(z)=\abs{D^\mu(g)(z)}^2 \frac{\mug(z)^2}{1-\abs{z}}\,dA(z)$, and
$\nm{V_{\mu,g}}_{H^p\to H^p} \asymp \|I\|_{H^p\to T_2^p(\nu_g) }.$ This together with Theorem~\ref{th:cohn}(i) finishes the proof.
\end{Prf}
\medskip 

\begin{Prf}{\em{Theorem~\ref{th:Vmucompact}}.}
By Theorem~\ref{th: caracterizacion BMOAmain} (ii)$\Leftrightarrow$(iii).
Next, by Theorem~\ref{th:main} $V_{\mu,g}$ is compact on $H^p$ if and only if $I: H^p\to   T_2^p(\nu_g)$ is compact where $d\nu_g(z)=\abs{D^\mu(g)(z)}^2 \frac{\mug(z)^2}{1-\abs{z}}\,dA(z)$.
This together with Theorem~\ref{th:cohn}(ii) finishes the proof.
\end{Prf}

\section{Fractional Volterra-type operator acting on classical  Bergman space.}\label{s6}
In order to prove the main results of the next sections we recall some elementary facts from the hyperbolic geometry of the unit disc. Let 
$\rho(z,w) 
=\left|\frac{z-w}{1-\overline{z}w}\right|,\, z,w\in\D,$ be 
 the  pseudohyperbolic metric of $\D$, and
$$\beta(z,w)=\frac{1}{2}\log\frac{1+\rho(z,w)}{1-\rho(z,w)},\qquad z,w\in\D,$$
 the hyperbolic metric of $\D$. We denote
 $\displaystyle{D(z,r)= \{w\in\D:  \beta(z,w)<r\}}$,   $z\in \D$, $r>0$. It is  well known, see \cite[Proposition 4.5]{Zhu}, that for any $z\in \D$ and $r>0$, there exist constants $C_1,C_2,C_3$, depending on $r$, such that
\begin{align}
	\frac{1}{C_1}|1-\overline{w}a|&\leq |1-\overline{w}z|\leq C_1|1-\overline{w}a|\,,\qquad w\in \D,\,a\in D(z,r),\nonumber\\
	\frac{1}{C_2}(1-|w|^2)&\leq 1-|z|^2\leq C_2(1-|w|^2)\,,\qquad w\in D(z,r),\label{geometric hyperbolic estimates}\\
	C_3|D(z,r)|&\leq (1-|z|^2)^2\leq C_3|D(z,r)|,\nonumber
\end{align}
where $|D(z,r)|$ denotes the Lebesgue area of the set $D(z,r)$. The estimates of the form \eqref{geometric hyperbolic estimates} will be used repeatedly in the the proofs of the next results and
will be called hyperbolic estimates. The next result can be found in
\cite[Theorem 7.4 and 7.7]{Zhu}.

\begin{lettertheorem}\label{th: carleson standard}
	Let $0<p<\infty$, $\a>-1$ and $\mu$ be a positive Borel measure on $\D$. Then,
	\begin{enumerate}
		\item[\rm(i)]
		$I: A^p_\a\to  L^p(\mu)$ is bounded 
		if and only if for all (some) $0<r<1$
		$$
		\sup_{a\in\D} \frac{\mu\left (D(a,r) \right )}{(1-\abs{a})^{2+\a}} <\infty.
		$$
		Moreover, $$\|I\|^p_{A^p_\a\to L^p(\mu) }\asymp \sup_{a\in\D} \frac{\mu\left (D(a,r) \right )}{(1-\abs{a})^{2+\a}};$$
		\item[\rm(ii)]
		$I: A^p_\a\to  L^p(\mu)$ is compact 
		if and only if for all (some) $0<r<1$
		$$
		\lim_{\abs{a}\to 1^-} \frac{\mu\left (D(a,r) \right )}{(1-\abs{a})^{2+\a}}=0.
		$$ 
	\end{enumerate}
\end{lettertheorem}
\begin{proposition}\label{propo: bloch fraccionaria}
	Let $\mu\in\DDD$, $0<p<\infty$, $\a>-1$ and $g\in\H(\D)$. Then 
	\begin{itemize}
	\item[\rm(i)] $g\in\B^\mu$ if and only if for all (some)  $0<r<1$
	$$\sup_{a\in\D}\frac{\int_{D(a,r)}\abs{D^\mu g(z)}^p\mug(z)^p(1-\abs{z})^{\a}dA(z)}{(1-\abs{a})^{\a+2}}<\infty,$$
Moreover, $$\| g\|^p_{\B}\asymp \sup_{a\in\D}\frac{\int_{D(a,r)}\abs{D^\mu g(z)}^p\mug(z)^p(1-\abs{z})^{\a}dA(z)}{(1-\abs{a})^{\a+2}}.$$
	\item[\rm(ii)] $g\in\B_0^\mu$ if and only if for every $0<r<1$
	$$\lim_{\abs{a}\to1^-}\frac{\int_{D(a,r)}\abs{D^\mu g(z)}^p\mug(z)^p(1-\abs{z})^{\a}dA(z)}{(1-\abs{a})^{\a+2}}=0.$$
	\end{itemize}
\end{proposition}
Proposition~\ref{propo: bloch fraccionaria} can be proved using standard arguments, so we omit its proof.
\begin{Prf}{\em{Theorem~\ref{th:bergmanVmuboundedcompact}}.}
The equivalences (ai)$\Leftrightarrow$(aii) and (bi)$\Leftrightarrow$(bii) follow from \cite[Theorem~3]{MorenoPelRosa}.
	First we notice that as $\mu\in\DDD$ and standard weights belongs to $\DDD$, by \cite[Theorem 1]{PelRosa}, $V_{\mu,g}$ is bounded (respect. compact) on $A^p_\a$ if and only if $I:A^p_\a\to L^p(\mu)$ is bounded (respect. compact) with $d\mu(z) = \abs{D^\mu g(z)}^p\mug(z)^p(1-\abs{z})^\a dA(z)$ and 
$\| V_{\mu,g}\|_{A^p_\alpha\to A^p_\alpha}\asymp \| I\|_{A^p_\alpha\to L^p(\mu)}.$
Therefore, the result follows from Theorem~\ref{th: carleson standard} and Proposition~\ref{propo: bloch fraccionaria}.
\end{Prf}

\section{Fractional derivative description of Besov spaces.} \label{s7}

Aiming to prove Theorem~\ref{th:schattendugintro} we will obtain fractional derivative description of classical Besov spaces which are  of its own interest. Despite the fact that the proof uses ideas previously employed to show Theorems \ref{th:main} and \ref{th: caracterizacion BMOAmain},
it contains a good number of technicalities due to the nature of $B_p$ spaces $0<p<\infty$.

\begin{lemma}\label{le:qp}
Let $\eta$ be a radial weight  and $0<p\le q<\infty$. Then,
$$ \|f\|_{A^q_{\nu(q,p)}}\lesssim \|f\|_{A^p_\eta}, \quad f \in \H(\D),$$
where $\nu(q,p)(|z|)=\eta(|z|)\left( \widehat{\eta}\left(\frac{1+|z|}{2}\right)(1-|z|)\right)^{\frac{q-p}{p}}.$
\end{lemma}
\begin{proof}
Joining the inequality $M_p(r,f)\le \frac{\|f\|_{A^p_\eta}}{\widehat{\eta}(r)^\frac{1}{p}}, 0\le r<1$, with the inequality
$$M_\infty(r,f)\le C_p \frac{M_p(s,f)}{(s-r)^{ \frac{1}{p}}}, \quad 0\le r<s<1,$$
it follows that
\begin{equation*}\begin{split}
M_\infty(r,f)\le C_p \frac{M_p\left(\frac{1+r}{2},f \right)}{(1-r)^{ \frac{1}{p}}}\le C_p 
\frac{\|f\|_{A^p_\eta}}{ \left( (1-r)\widehat{\eta}\left( \frac{1+r}{2} \right)\right)^\frac{1}{p} },
\quad 0\le r<1.
\end{split}\end{equation*}
Therefore,
\begin{equation*}\begin{split}
\|f\|^q_{A^q_{\nu(q,p)}} &=\int_{\D} |f(z)|^p |f(z)|^{q-p} \eta(|z|)\left( \widehat{\eta}\left(\frac{1+|z|}{2}\right)(1-|z|)\right)^{\frac{q-p}{p}}\, dA(z)
\\ & \lesssim \|f\|^{q-p}_{A^p_{\eta}}\int_{\D} |f(z)|^p \eta(z) \, dA(z)=\|f\|^q_{A^p_\eta}, \quad f \in \H(\D).
\end{split}\end{equation*}
This finishes the proof.
\end{proof}

\begin{theorem}\label{th:Bpmuradial}
   Let  $1\le p<\infty$ and $\mu$ be  a radial weight. Then, $B_{p,\mu}\subset B_p$ and 
    $$ \| f\|^p_{B_p}\lesssim \| f\|^p_{B_{p,\mu}}, \quad f\in \H(\D). $$
Moreover, if $\mu\in \DD$ and $0<p<1$, then $B_{p,\mu}\subset B_p$ and 
  $$ \| f\|^p_{B_p}\lesssim \| f\|^p_{B_{p,\mu}}, \quad f\in \H(\D). $$
\end{theorem}
\begin{proof}
	Throughout this proof we will use the notation $d\lambda(z)=\frac{dA(z)}{(1-|z|^2)^2}$. \\
Since
$\lim_{r\to 1^-} \| f-f_r\|^p_{B_p}=\lim_{r\to 1^-} \| f-f_r\|^p_{B_{p,\mu}}=0$, so it is enough to prove the result for $f\in \H(\overline{\D})$.
If $f\in \H(\overline{\D})$,
 by Corollary~ \ref{prop: expresión derivada fraccionaria}
			$$
			D^{\mu}(f)(z) = \sum_{j=0}^{n-1} \frac{\widehat{f}(j)}{\mu_{2j+1}}z^j + 4^n z^n R^{W^n,V_{\mu^+,n}}(f^{(n)})(z),\quad z\in\D.
			$$
		Then, 	if we write $g=D^\mu(f)$ 
			$$
			R^{V_{\mu^+,n},W^n}(g)(z) = \sum_{j=0}^{n-1}\frac{(V_{\mu^+,n})_{2j+1}}{(W^n)_{2j+1}\mu_{2j+1}}\widehat{f}(j)z^j + 4^n z^n f^{(n)}(z), \quad z\in \D.
			$$

From now on we will split the proof in two cases.

{\bf{Case $\mathbf{1<p<\infty}$.}}
Bearing in mind the  Littlewood-Paley formula for standard Bergman spaces and  the subharmonicity of $|f|^q$,  
for each  $0<q<\infty$,  $n\in \N$,  and $0<r<1$
\begin{equation}\begin{split}\label{eq:bpmu1}
 &\| f\|^p_{B_p}\asymp  \sum_{j=0}^{n-1} | f^{(j)}(0)|^p+ \int_{\D} |f^{(n)}(z)|^p(1-|z|)^{np-2}\,dA(z)
\\&  \asymp  \sum_{j=0}^{n-1} | f^{(j)}(0)|^p+ \int_{\D} | z^nf^{(n)}(z)|^p(1-|z|)^{np-2}\,dA(z)
\\ & \lesssim \sum_{j=0}^{n-1} | f^{(j)}(0)|^p+  \int_{\D} \left( \int_{D(z,r)} |u^nf^{(n)}(u)|^q(1-|u|)^{qn-2}\, dA(u)\right)^{\frac{p}{q}} \,d\lambda(z)
\\ & = \sum_{j=0}^{n-1} | f^{(j)}(0)|^p
\\ &+  \int_{\D} \left( \int_{D(z,r)} \left| R^{V_{\mu^+,n},W^n}(g)(u) - \sum_{j=0}^{n-1}\frac{(V_{\mu^+,n})_{2j+1}}{(W^n)_{2j+1}\mu_{2j+1}}\widehat{f}(j)u^j  \right|^q(1-|u|)^{qn-2}\, dA(u)\right)^{\frac{p}{q}} \, d\lambda(z).
\\ & \lesssim \| f\|^p_{B_{p,\mu}} + \int_{\D} \left( \int_{D(z,r)} \left| R^{V_{\mu^+,n},W^n}(g)(u) \right|^q(1-|u|)^{qn-2}\, dA(u)\right)^{\frac{p}{q}} \, d\lambda(z).
\end{split}\end{equation}
  By Theorem~\ref{th:perala} and Lemma~\ref{lema pesos tec}(iii)
    \begin{equation*}
				\begin{split}
					\abs{R^{V_{\mu^+,n},W^n}(g)(u)}^q 
					&\lesssim \left (\int_\D \abs{g(\z)}\abs{B^{W^n}_u(\z)}\mug(\z)(1-\abs{\z})^{n}dA(\z) \right )^q,\quad u\in\D.
				\end{split}
			\end{equation*}
Now fix $1<q<\min\{2,p\}$ and $\e\in (0,1)$. Then take $n\in \N$ such that $n>1+\frac{1}{p}+ \e\left(\frac{1}{q}-\frac{1}{p}\right)$, so 
$2n+2-\frac{\e}{q}-\frac{2}{q'}>n+1+\frac{2}{q}+\frac{1-\e}{p}$ and therefore we can choose $N\in \N$ such that
         $$2n+2 -\frac{\e}{q}-\frac{2}{q'}>N>n+\frac{2}{q}+\frac{1-\e}{p}.$$ This choice of the paremeters will be employed throughout  the rest of the proof.
           
          Now, by H\"older's inequality and \cite[Lemma 5]{PelRatBMO}, for $u\in \D$
            \begin{align*}
                	\abs{R^{V_{\mu^+,n},W^n}(g)(u)}^q 
					&\leq \left (\int_\D \frac{\abs{g(\z)}^q}{\abs{1-\overline{u}\z}^{qN}}\mug(\z)^q dA_{qn-\e}(\z) \right) \left (\int_{\D}\abs{(1-\overline{u}\z)^NB^{W^n}_u(\z)}^{q'} dA_{\e q'/q}(\z)\right )^{\frac{q}{q'}} \\
					&\lesssim 
\left (\int_\D \frac{\abs{g(\z)}^q}{\abs{1-\overline{u}\z}^{qN}}\mug(\z)^q dA_{qn-\e}(\z) \right)
\left (1+ \int_0^{\abs{u}}\frac{(1-r)^{\e q'/q+1} dr}{\widehat{W^{n}}(r)^{q'}(1-r)^{q'(1-N)}} \right )^{\frac{q}{q'}}\,.
            \end{align*}
Now, since $(2n+2-N- \frac{\e}{q})q'>2$,  arguing as in the proof of \eqref{eq: th1 desig Wn} we deduce that
$$1+ \int_0^{\abs{u}}\frac{(1-r)^{\e q'/q+1} dr}{\widehat{W^{n}}(r)^{q'}(1-r)^{q'(1-N)}}\lesssim  \frac{1}{(1-|u|^2)^{(2n+2-N-\frac{\e}{q})q'-2}},\quad u\in\D.$$
Therefore,
\begin{equation*}\begin{split}
\abs{R^{V_{\mu^+,n},W^n}(g(u)}^q \lesssim \left (\int_\D \frac{\abs{g(\z)}^q}{\abs{1-\overline{u}\z}^{qN}}\mug(\z)^q dA_{qn-\e}(\z) \right)
\frac{1}{ (1-|u|^2)^{(2n+2-N-\frac{\e}{q})q-2\frac{q}{q'}}},\quad u\in\D,
\end{split}\end{equation*}
and by Fubini's theorem and the hyperbolic estimates
\begin{equation*}\begin{split}
& \int_{D(z,r)} \left| R^{V_{\mu^+,n},W^n}(g)(u) \right|^q(1-|u|)^{qn-2}\, dA(u)
\\ & \lesssim   \int_{D(z,r)} \left (\int_\D \frac{\abs{g(\z)}^q}{\abs{1-\overline{u}\z}^{qN}}\mug(\z)^q dA_{qn-\e}(\z) \right)
\frac{1}{ (1-|u|^2)^{(n-N)q-\e+4}}\, dA(u)
\\ & = \int_\D \abs{g(\z)}^q \mug(\z)^q  \int_{D(z,r)} \frac{dA(u)}{\abs{1-\overline{u}\z}^{qN} (1-|u|^2)^{(n-N)q-\e+4}}\,      dA_{qn-\e}(\z) 
\\ & \asymp  \frac{1}{ (1-|z|^2)^{(n-N)q-\e+2}} \int_\D \frac{\abs{g(\z)}^q}{\abs{1-\overline{z}\z}^{qN}} \mug(\z)^q dA_{qn-\e}(\z).
\end{split}\end{equation*}
The above estimate together with \eqref{eq:bpmu1} implies that

\begin{equation}\begin{split}\label{eq:bpmu2}
 \| f\|^p_{B_p} 
& \lesssim \| f\|^p_{B_{p,\mu}} + \int_{\D} \left(  \int_\D \frac{\abs{g(\z)}^q}{\abs{1-\overline{z}\z}^{qN}} \mug(\z)^q dA_{qn-\e}(\z)\right)^{\frac{p}{q}} \, \frac{dA(z)}{ (1-|z|^2)^{(n-N)p+(2-\e)\frac{p}{q}+2} }.
\end{split}\end{equation}
Next, by H\"older's inequality and \cite[Lemma~3.10]{Zhu}, it follows that
\begin{equation*}\begin{split}
 &\left(  \int_\D \frac{\abs{g(\z)}^q}{\abs{1-\overline{z}\z}^{qN}} \mug(\z)^q dA_{qn-\e}(\z)\right)^{\frac{p}{q}}
   \\ & \lesssim  \int_\D \frac{\abs{g(\z)}^p}{\abs{1-\overline{z}\z}^{pN-2\left( \frac{p}{q}-1\right) }} \mug(\z)^p dA_{pn-\e}(\z)          \left(  \int_\D \frac{dA_{-\e}(\z)}{\abs{1-\overline{z}\z}^{2}}  \right)^{\frac{p}{q}-1}
\\ & \asymp  \frac{1}{(1-|z|^2)^{\e\left(\frac{p}{q}-1 \right)}} \int_\D \frac{\abs{g(\z)}^p}{\abs{1-\overline{z}\z}^{pN-2\left( \frac{p}{q}-1\right) }} \mug(\z)^p dA_{pn-\e}(\z).
\end{split}\end{equation*}
Finally, bearing in mind that  $(n-N)p+2\frac{p}{q}+2-\ep<1$ and $n>\frac{\ep-2}{p}$, the above estimate
 together with \eqref{eq:bpmu2}, Fubini's theorem and a another application of \cite[Lemma~3.10]{Zhu} yields
\begin{equation*}\begin{split}
 \| f\|^p_{B_p}
& \lesssim 
\| f\|^p_{B_{p,\mu}} + \int_\D \int_\D \frac{\abs{g(\z)}^p}{\abs{1-\overline{z}\z}^{pN-2\left( \frac{p}{q}-1\right) }} \mug(\z)^p dA_{pn-\e}(\z)
    \frac{dA(z)}{ (1-|z|^2)^{(n-N)p+2\frac{p}{q}+2-\ep} }
\\ & =  \| f\|^p_{B_{p,\mu}} +  \int_\D \abs{g(\z)}^p  \mug(\z)^p   \int_{\D}   \frac{dA(z)}{ (1-|z|^2)^{(n-N)p+2\frac{p}{q}+2-\ep}  \abs{1-\overline{z}\z}^{pN-2\left( \frac{p}{q}-1\right) }}     dA_{pn-\e}(\z)
\\ & \asymp 
 \| f\|^p_{B_{p,\mu}}.
\end{split}\end{equation*}

{\bf{Case $\mathbf{0<p\le 1}$.}} Take $n\in \N$ such that $np>2-2p$. A similar argument to that of the previous case (but simpler) implies that
\begin{equation}\begin{split}\label{eq:bpmurad1}
 \| f\|^p_{B_p} & \lesssim  \| f\|^p_{B_{p,\mu}} + 
 \int_{\D}\left| R^{V_{\mu^+,n},W^n}(g)(u) \right|^p(1-|u|)^{pn-2} \,dA(u).
\\ & \lesssim  \| f\|^p_{B_{p,\mu}} + 
 \int_{\D}\left (\int_\D \abs{g(\z)}\abs{B^{W^n}_u(\z)}\mug(\z)(1-\abs{\z})^{n}dA(\z) \right )^p (1-|u|)^{pn-2} \,dA(u).
\end{split}\end{equation}
If $p=1$, using  that $\widehat{W^n}(r)\asymp (1-r)^{2n+1}$ by Lemma~\ref{lema pesos tec},  $W^n, (1-|u|)^{pn-2}\in \DDD$,  and  \cite[Theorem~1]{PR2016/1}, it follows that
\begin{equation*}\begin{split}
 \| f\|_{B_1}
& \lesssim \| f\|_{B_{1,\mu}}+ \int_\D \abs{g(\z)}\mug(\z)(1-\abs{\z})^{n}\int_{\D} \abs{B^{W^n}_u(\z)} (1-|u|)^{n-2} \,dA(u) \, dA(\z)  
\\ & \asymp \| f\|_{B_{1,\mu}}+\int_\D \abs{g(\z)}\mug(\z)(1-\abs{\z})^{n} \left(1+\int_0^{|\z|} \frac{dt}{(1-t)^{n+3}} \right) \, dA(\z)  
\\ & \asymp \| f\|_{B_{1,\mu}}, \quad f \in \H(\D).
\end{split}\end{equation*}
Next if $0<p<1$ and $\mu\in \DD$,  take $\eta(s)=\mug(s)^p(1-s)^{(n+2)p-2}$. Then $\eta \in \DDD$ and  
\begin{equation*}\begin{split}
\eta(|\z|)\left( \widehat{\eta}\left(\frac{1+|\z|}{2}\right)(1-|\z|)\right)^{\frac{1-p}{p}}
& \asymp 
\eta(|\z|)\left( \widehat{\eta}\left(|\z|\right)(1-|\z|)\right)^{\frac{1-p}{p}}
\\ & \asymp \mug(\z)(1-\abs{\z})^{n}, \z\in \D.
\end{split}\end{equation*}
So by \eqref{eq:bpmurad1},  Lemma~\ref{le:qp}  and \cite[Theorem~1]{PR2016/1},  and bearing in mind that $np>2-2p$
\begin{equation*}\begin{split}
 \| f\|^p_{B_p} & \lesssim  \| f\|^p_{B_{p,\mu}} + 
 \int_{\D} \int_\D \abs{g(\z)}^p\abs{B^{W^n}_u(\z)}^p\mug(\z)^p(1-\abs{\z})^{np+2p-2} \,dA(\z)  (1-|u|)^{pn-2} \,dA(u)
\\ & =   \| f\|^p_{B_{p,\mu}} +  \int_\D \abs{g(\z)}^p\mug(\z)^p (1-\abs{\z})^{np+2p-2}\int_{\D} \abs{B^{W^n}_u(\z)}^p (1-|u|)^{pn-2} \,dA(u) \, dA(\z) 
\\ & \asymp \| f\|^p_{B_{p,\mu}} +  \int_\D \abs{g(\z)}^p\mug(\z)^p (1-\abs{\z})^{np+2p-2} 
\left(1+\int_0^{|\z|} \frac{dt}{(1-t)^{np+2p+1}} \right) \, dA(\z) 
\\ & \asymp \| f\|^p_{B_{p,\mu}},  \quad f \in \H(\D).
\end{split}\end{equation*}
This finishes the proof.
\end{proof}

\begin{theorem}\label{th:BpmuD}
 Let  $1\le p<\infty$ and $\mu$ be  a radial weight. Then, the following conditions are equivalent:
\begin{enumerate}
\item[\rm(i)] $\frac{\widehat{\mu}(z)^p}{(1-|z|)^2} \in \mathcal{D}$;  
\item[\rm(ii)] \begin{equation}\label{eq:bpmuds}
	 \|g\|_{B_{p,\mu}} \asymp \|g\|_{B_p},\quad g \in\mathcal{H}(\mathbb{D}).
	\end{equation}
\item[\rm(iii)] $$
	 \|g\|_{B_{p,\mu}} \lesssim \|{g}\|_{B_p},\quad g\in\mathcal{H}(\mathbb{D});
	$$
\end{enumerate}
Moreover, if $0<p\le 1$ and $\mu\in \DD$, the previous three conditions are also equivalent.
\end{theorem}
It is also worth mentioning that the proof of Theorem~\ref{th:BpmuD} reveals that   (i) implies (ii) for any radial weight and $0<p\le 1$.

\begin{proof}

Assume that (i) holds. In order to prove (ii) we will split the proof in several steps.

{\bf{First step.}} Let us prove that $\mu\in \DD$. 

The class $\DDD$ is closed under multiplication by weights of the form $(1-s)^\beta$, $\beta>0$. Therefore choosing $\beta=2$ it follows that $\widehat{\mu}^p\in \DDD$, and in particular 
$\widehat{\mu}^p\in \DD$. Now bearing in mind the identity $\om_x=x(\omg)_{x-1}, x\ge 1$, and  Lemma~\ref{lemma: caract Dgorro},  it follows  that a radial weight $\om\in\DD$ if and only if $\omg\in\DD$.
Therefore, the radial weight $ p\mu(z)\widehat{\mu}(z)^{p-1}\in \DD$, which is equivalent to the fact that $\mu\in \DD$.

{\bf{Second step.}} Let us prove that 
\begin{equation}\label{eq:bpmud1}
\int_{r}^1 \frac{\widehat{\mu}(s)^p}{(1-s)^2}\,ds \asymp \frac{\widehat{\mu}(r)^p}{1-r}, \quad 0\le r<1.
\end{equation} 
Since $\mu\in\DD$, 
$$  \int_{r}^1 \frac{\widehat{\mu}(s)^p}{(1-s)^2}\,ds\ge \int_{r}^{\frac{1+r}2} \frac{\widehat{\mu}(s)^p}{(1-s)^2}\,ds
\gtrsim  \frac{\widehat{\mu}(r)^p}{1-r}, \quad 0\le r<1.
$$
On the other hand since $\frac{\widehat{\mu}(z)^p}{(1-|z|)^2} \in \Dd$, by \cite[Lemma B]{PelRosa} there is $K>1$ and $C>0$ such that
\begin{equation*}\begin{split}
C \int_{r}^1 \frac{\widehat{\mu}(s)^p}{(1-s)^2}\,ds \le \int_{r}^{1-\frac{1-r}{K}} \frac{\widehat{\mu}(s)^p}{(1-s)^2}\,ds\le  (K-1)\frac{\widehat{\mu}(r)^p}{1-r},
\end{split}\end{equation*}
and therefore \eqref{eq:bpmud1} holds.

{\bf{Third step.}} Let us prove that  $\mu\in \Dd$. 

By Lemma~\ref{lemma: caract dcheck} (ii) and \eqref{eq:bpmud1}
 there is $\b>0$ and $C>0$ such that
$$ \frac{\widehat{\mu}(s)^p}{(1-s)^{1+\beta}}\le C \frac{\widehat{\mu}(t)^p}{(1-t)^{1+\beta}}, \quad 0\le t\le s<1.$$
Therefore
$$ \int_0^s \frac{dt}{\widehat{\mu}(t)^p(1-t)}\le C \frac{(1-s)^{1+\beta}}{\widehat{\mu}(s)^p} \int_0^s \frac{dt}{(1-t)^{2+\beta}}\lesssim \frac{1}{\widehat{\mu}(s)^p}, \quad 0\le s<1, $$
which together with  a new application of Lemma~\ref{lemma: caract dcheck}  implies that $\mu\in \Dd$.

{\bf{Fourth step.}} Let us prove the identity \eqref{eq:bpmuds}. Fix $n\in\N$ such that $np>1$. Since $\mu$, $(1-|z|)^{np-2}$ and $\frac{\widehat{\mu}(z)^p}{(1-|z|)^2}$ are radial doubling weights, by \cite[Theorem~1]{PelRosa}, \cite[Theorem~5]{PR19} and
the indentity 
\begin{equation}\label{eq:F}
D^\mu(F_n)(z)=z^n(D^\mu(f))^{(n)}(z), \quad  F_n(\zeta)=\zeta^n f^{(n)}(\z),\, z, \z\in \D,\quad n\in \N,
\end{equation}
it follows that
\begin{equation*}\begin{split}
\|f\|_{B_p}^p  &\asymp \sum_{j=0}^{n-1}|f^{(j)}(0)|^p+ \int_{\D} |F_n(z)|^p(1-|z|)^{np-2}\,dA(z)
\\ & \asymp \sum_{j=0}^{n-1}|f^{(j)}(0)|^p + \int_{\D} |D^\mu(F_n)(z)|^p \widehat{\mu}^p(z) (1-|z|)^{np-2}\,dA(z)
\\ & = \sum_{j=0}^{n-1}|f^{(j)}(0)|^p+ \int_{\D} |z|^{np} |(D^\mu(f))^{(n)}(z)|^p \widehat{\mu}^p(z) (1-|z|)^{np-2}\,dA(z)
\\ & \asymp \sum_{j=0}^{n-1}|f^{(j)}(0)|^p + \int_{\D}  |(D^\mu(f))^{(n)}(z)|^p \widehat{\mu}^p(z) (1-|z|)^{np-2}\,dA(z)
\\ & \asymp  \|f\|_{B_{p,\mu}}^p.
\end{split}\end{equation*}

Now let us prove (ii)$\Rightarrow(i)$.

{\bf{First step.}} Let us prove that $\mu\in \DD$ if $1<p<\infty$.

By choosing $f(z)=z^n$ in \eqref{eq:bpmuds}, it follows that 
$$ \int_0^1 s^{np+1} \frac{\widehat{\mu}(s)^p}{(1-s)^2}\,ds \lesssim n (\mu_{2n+1})^p, \quad n\in \N\cup\{0\},$$
and therefore
\begin{equation}\label{eq:bpmud4}
\int_0^1 s^{x} \frac{\widehat{\mu}(s)^p}{(1-s)^2}\,ds \lesssim x (\mu_{\frac{2x}{p}})^p, \quad x\ge 1.
\end{equation}
So for any $t\in (0,1)$ 
\begin{equation*}\begin{split}
(\mu_x)^p & \asymp x^p (\widehat{\mu}_x)^p
\\ & \le   x^p \int_0^1 s^{tpx} \frac{\widehat{\mu}(s)^p}{(1-s)^2}\,ds \left(  \int_0^1 s^{(1-t)p'x} (1-s)^{\frac{2p'}{p}}\,ds \right)^{\frac{p}{p'}}
\\ & \asymp x^{-1} \int_0^1 s^{tpx} \frac{\widehat{\mu}(s)^p}{(1-s)^2}\,ds
\lesssim (\mu_{2tx})^p,\quad  x\ge 1.
\end{split}\end{equation*}
Therefore, taking $t=3/4$ and iterating the obtained inequality
$$\mu_x\lesssim \mu_{\frac{3x}{2}}\lesssim  \mu_{\frac{9x}{4}}\le \mu_{2x}, \quad  x\ge 1,$$
so $\mu\in \DD$ by Lemma~\ref{lemma: caract Dgorro}.

{\bf{Second step.}} Let us prove that $\mu\in \Dd$. If $K>1$, 
\begin{equation*}\begin{split}
\int_0^1 s^{x} \frac{\widehat{\mu}(s)^p}{(1-s)^2}\,ds & \ge \int_{1-\frac{1}{x}}^{1-\frac{1}{Kx}} s^{x} \frac{\widehat{\mu}(s)^p}{(1-s)^2}\,ds
 \ge C_1  \widehat{\mu}\left(1-\frac{1}{Kx}\right)^p(K-1)x, \quad x\ge 2,
\end{split}\end{equation*}
where $C_1=\inf_{x\ge 2}\left( 1-\frac{1}{x}\right)^x>0$.
The above inequality, together with \eqref{eq:bpmud4}, the fact that $\mu\in \DD$ and Lemma~\ref{lemma: caract Dgorro}, implies that there is $C_2=C_2(\mu)>0$ such that
\begin{equation*}\begin{split}
C_1  \widehat{\mu}\left(1-\frac{1}{Kx}\right)^p(K-1)x &\le  \int_0^1 s^{x} \frac{\widehat{\mu}(s)^p}{(1-s)^2}\,ds 
 \le C_2 x \widehat{\mu}\left(1-\frac{1}{x} \right) ^p, \quad x\ge 2.
\end{split}\end{equation*}
That is,
$$\widehat{\mu}(s)\ge \left( \frac{(K-1)C_1}{C_2}\right)^{1/p}\widehat{\mu}\left(1 - \frac{1-s}{K}\right), \quad \frac{1}{2}\le s<1.$$
So, choosing $K>\frac{C_2}{C_1}+1$, it follows that $\mu\in \Dd$.
 
{\bf{Third step.}} Let us prove that  $\frac{\widehat{\mu}(z)^p}{(1-|z|)^2}\in \DDD$. 

Fix $n\in \N$ such that $np>1$.  Since $\mu$ and $(1-|z|)^{np-2}$ are radial doubling weights, by \eqref{eq:F},  \cite[Theorem~1]{PelRosa} and  \eqref{eq:bpmuds}
\begin{equation*}\begin{split}
&\sum_{j=0}^{n-1}|(D^\mu(f))^{(j)}(0)|^p+ \int_{\D}  |(D^\mu(f))^{(n)}(z)|^p \widehat{\mu}^p(z) (1-|z|)^{np-2}\,dA(z)
\\ & \asymp  \sum_{j=0}^{n-1}|f^{(j)}(0)|^p+ \int_{\D} |z|^{np} |(D^\mu(f))^{(n)}(z)|^p \widehat{\mu}^p(z) (1-|z|)^{np-2}\,dA(z)
\\ & = \sum_{j=0}^{n-1}|f^{(j)}(0)|^p + \int_{\D} |D^\mu(F_n)(z)|^p \widehat{\mu}^p(z) (1-|z|)^{np-2}\,dA(z)
\\ & \asymp \sum_{j=0}^{n-1}|f^{(j)}(0)|^p+ \int_{\D} |F_n(z)|^p(1-|z|)^{np-2}\,dA(z)
\\ & \asymp \|f\|_{B_p}^p 
 \asymp \int_{\D}  |D^\mu(f)(z)|^p \widehat{\mu}^p(z) (1-|z|)^{-2}\,dA(z),\quad f\in \H(\D).
\end{split}\end{equation*}
That is, 
\begin{equation*}\begin{split}
& \sum_{j=0}^{n-1}|g^{(j)}(0)|^p+ \int_{\D}  |g^{(n)}(z)|^p \widehat{\mu}^p(z) (1-|z|)^{np-2}\,dA(z)
\\ & \asymp
 \int_{\D}  |g(z)|^p \widehat{\mu}^p(z) (1-|z|)^{-2}\,dA(z),\quad g\in \H(\D).
\end{split}\end{equation*}
So, by \cite[Theorem~5]{PR19} $\frac{\widehat{\mu}(z)^p}{(1-|z|)^2}\in \DDD$. 


The equivalence between (ii) and (iii) follows from Theorem~\ref{th:Bpmuradial}. 

\end{proof}

\section{Schatten p-classes of Fractional Volterra-type operator.}\label{s8}
Throughout this section we shall use the notation $A^2_{-1}=H^2$ and $dA_{-1}(z)=\frac{dA(z)}{1-|z|^2}$.
We recall that a  sequence $\{z_\lambda\}_\lambda\subset \D$ is called $r$-hyperbolically  separated if there exists  a constant $r>0$ such that $\beta(z_k,z_\lambda)\geq r $ for $k\neq\lambda$, while is said to be an $r$-lattice in the hyperbolic distance, if it is $r/2$-separated and  
$$\mathbb{D} = \bigcup_k D(z_k, r).$$
 Now let us recall some results on singular values and  Schatten classes which will be used in the proofs of our 
results. In fact, 
a characterization of the singular values is given by the Rayleigh's equation. Let $n\in\N$, and $\lambda_n(T)$ be the $n$-th singular value of the positive operator $T$. Then,
\begin{equation}\label{Min-Max Theorem}
    \lambda_{n+1}=\min_{x_1,\dots,x_n}\max\{\langle Tx,x\rangle\colon \|x\|=1, x\perp x_i, 1\leq i\leq n\}\,. 
\end{equation}
A proof of this result can be found in \cite[Ch. X.4.3]{DS}.
On the other hand, the Spectral Mapping theorem imply that the condition $T\in S^p(\mm{H})$ is equivalent to $T^*T\in S^{\frac{p}{2}}(\mm{H})$ and
\begin{equation}\label{equation T*T norm}
    \|T\|_p^p=\|T^*T\|_{S^{p/2}}^{p/2}\,
\end{equation}
see \cite[Theorem~1.26]{Zhu}.

\begin{theorem}\label{th:schattendug}
Let $\mu\in \DDD$, $\alpha\ge -1$, $0<p<\infty$ and $g\in \H(\D)$. Then,
\begin{enumerate}
\item[\rm(i)] If $\int_0^1 \frac{\mug(r)^p}{(1-r)^2}\,dr=+\infty$, then $V_{\mu,g}\in S_p(A^2_\alpha)$ if and only if $g=0$.
 \item[\rm(ii)] If $\frac{\mug(r)^p}{(1-r)^2}$ is a radial weight, then $V_{\mu,g}\in S_p(A^2_\alpha)$ if and only if $g\in B_{p,\mu}$, and
$$\| V_{\mu,g}\|_{S_p(A^2_\alpha)}\asymp \|g\|_{B_{p,\mu}}.$$
\end{enumerate}
\end{theorem}
\begin{proof} 
If $g\in B_{p,\mu}$, 
by Theorem~\ref{th:BpmuD} $g\in B_{p}$ and therefore  $g\in \VMOA$ \cite[Lemma 10.13]{Zhu} and by Theorem~\ref{th: caracterizacion BMOAmain} if $\a=-1$, Proposition~\ref{propo: bloch fraccionaria} if $\a>-1$,  $|D^\mu(g)(z)|^2 \widehat{\mu}(z)^2 dA_\alpha(z)$ is a Carleson measure for $A^2_\alpha$. On the other hand,
if $V_{\mu,g}\in S_p(A^2_\alpha)$, then $V_{\mu,g}$ is bounded on $A^2_\alpha$  and by Theorem~\ref{th:Vmubounded} if $\a=-1$, Theorem~\ref{th:bergmanVmuboundedcompact} if $\a>-1$, $|D^\mu(g)(z)|^2 \widehat{\mu}(z)^2 dA_\alpha(z)$ is a Carleson measure for $A^2_\alpha$.
\par So,  we may assume that $|D^\mu(g)(z)|^2 \widehat{\mu}(z)^2 dA_\alpha(z)$ is a Carleson measure for $A^2_\alpha$. Let $\{z_j\}_j\subset \D$ be an $r$-lattice on $\D$. We wiil prove that
\begin{equation}\label{eq 1}
\| V_{\mu,g}\|^p_{S_p(A^2_\alpha)}\asymp \sum_{\lambda=1}^{\infty}\left(\frac{1}{(1-|z_\lambda|^2)^2}\int_{D(z_\lambda,r)}|D^\mu(g)(z)|^2\widehat{\mu}(z)^2dA(z)\right)^{p/2}.
\end{equation}
First of all, we define the  sesquilinear form
$$\mathcal{V}_g(f,h):=\int_{\D}f(z)\overline{h(z)}|D^\mu(g)(z)|^2 \widehat{\mu}(z)^2 dA_\alpha(z).$$

Since $|D^\mu(g)(z)|^2 \widehat{\mu}(z)^2 dA_\alpha(z)$ is a Carleson measure for $A^2_\alpha$, 
\begin{align*}
    |\mathcal{V}_g(f,h)|^2&=\left|\int_{\D}f(z)D^\mu(g)(z)\overline{h(z)D^\mu(g)(z)}\widehat{\mu}(z)^2 dA_\alpha(z)\right|^2\\
    &\leq \int_{\D}|f(z)|^2|D^\mu(g)(z)|^2\widehat{\mu}(z)^2 dA_\alpha(z) \int_{\D}|h(z)|^2|D^\mu(g)(z)|^2\widehat{\mu}(z)^2 dA_\alpha(z)\\
    &\lesssim \|f\|_{A^2_\alpha}^2\|h\|_{A^2_\alpha}^2\,, \quad f,h\in A^2_\alpha.
\end{align*}
That is $\mathcal{V}_g$ is a bounded sesquilinear form on $A^2_\alpha\times A^2_\alpha$.
By  Riesz representation theorem  there exists an operator $\mathcal{V}_g\colon A^2_\alpha\to A^2_\alpha$, (we denote it as the sesquilinear form) such that
$$\langle \mathcal{V}_g(f),h\rangle_{A^2_\alpha} =\mathcal{V}_g(f,h)\qquad f,h\in A^2_\alpha\,.$$
Based on the definition, it is easy to see that $\mathcal{V}_g(f,f)\geq 0$ and consequently, $ \mathcal{V}_g$ is a positive operator. Moreover by Theorem~\ref{th:main} if $\a=-1$, \cite[Theorem 1]{PelRosa} if $\a>-1$, for every $f\in A^2_\alpha$, we have that
\begin{align*}
    \langle \mathcal{V}_g(f),f\rangle_{A^2_\alpha} &=\int_{\D}|f(z)|^2 |D^\mu(g)(z)|^2\widehat{\mu}(z)^2 dA_\alpha(z) \asymp \|V_{\mu,g}(f)\|_{A^2_\alpha}^2=\langle (V_{\mu,g})^*V_{\mu,g}(f),f\rangle_{A^2_\alpha} .
\end{align*}
Thus, there exist two constants $C_1,C_2$ depending on $\mu$, such that
$$C_1\langle (V_{\mu,g})^*V_{\mu,g}(f),f\rangle_{A^2_\alpha} \leq \langle \mathcal{V}_g(f),f\rangle_{A^2_\alpha} \leq C_2\langle (V_{\mu,g})^*V_{\mu,g}(f),f\rangle_{A^2_\alpha} 
\qquad f\in A^2_\alpha \,.$$
An application of min-max theorem of eigenvalues of positive operators and the fact that $\|V_{\mu,g}\|_{S^p}^p=\| (V_{\mu,g})^*V_{\mu,g}\|_{S^{p/2}}^{p/2}$ imply that
$$V_{\mu,g}\in S^p(A^2_\alpha)\iff \mathcal{V}_g\in S^{p/2}(A^2_\alpha)$$
and the corresponding Schatten norms are comparable with constants depending only on $\mu$.

Next, let is observe that if $B^\alpha_z$ is the reproducing kernel of $A^2_\alpha$, we have that
\begin{align*}
    \mathcal{V}_g(f)(z)&=\langle \mathcal{V}_g(f), B^\alpha_z\rangle_{A^2_\alpha}\\
    &= \mathcal{V}_g(f,B^\alpha_z)\\
    &= \int_{\D}f(w)|D^\mu(g)(w)|^2\overline{ B^\alpha_z(w)}\widehat{\mu}(z)^2 dA_\alpha(z)\\
    &=\int_{\D}f(w)\frac{1}{(1-\overline{w}z)^{2+\alpha}}d\mu_g(z)\\
    &= T_{\mu_g}(f)(z)\,,
\end{align*}
that is $ \mathcal{V}_g$ is nothing else but the Toeplitz operator induced by the measure  $d\mu_{g,\alpha}(z)=|D^\mu(g)(z)|^2 \widehat{\mu}(z)^2 dA_\alpha(z) $. Therefore, by \cite[Theorem p. 347]{Lu87} and the finite overlapping property of the hyperbolic discs $D(z_j,r)$ of a $r$-lattice $\{z_j\}_{j=1}^\infty$ (see \cite[Lemma~4.7]{Zhu}) it follows that
$$\|\mathcal{V}_g\|^{p/2}_{S^{p/2}(A^2_\alpha)}\asymp  \sum_{j=1}^{\infty}\left(\frac{\mu_{g,\alpha}(D(z_j,r))}{(1-|z_j|^{2+\alpha}}\right)^{p/2}$$
for any $r$-lattice $\{z_j\}_{j=1}^\infty$ .
Since
$$\mu_{g,\alpha}(D(z_j,r))=\int_{D(z_j,r)}|D^\mu(g)(z)|^2\widehat{\mu}(z)^2 dA_\alpha(z) \asymp (1-|z_j|^2)^{\alpha}\int_{D(z_j,r)}|D^\mu(g)(z)|^2\widehat{\mu}(z)^2dA(z)\,,$$
\eqref{eq 1} holds.

\par Observe that  if $\int_0^1 \frac{\mug(r)^p}{(1-r)^2}\,dr=+\infty$, 
$$\int_{\D}|D^\mu(g)(z)|^p\widehat{\mu}(z)^p\frac{dA(z)}{(1-|z|^2)}<\infty  \iff D^\mu(g)=0\iff g=0.$$
Therefore, it is left to prove that the condition on the lattice implies the integral condition on $\D$, that is
\begin{equation}
\label{eq 2}
   \int_{\D}|D^\mu(g)(z)|^p\widehat{\mu}(z)^p\frac{dA(z)}{(1-|z|^2)}\asymp\sum_{j=1}^{\infty}\left(\frac{\mu_g(D(z_j,r))}{(1-|z_j|^2)}\right)^{p/2}.
\end{equation}

\par Since $\mu\in \DD$, for each $r \in (0,1)$ 
\begin{equation}\label{eq:mugconstant}
\mug(z)\asymp \mug(w), \quad w \in D(z,r),
\end{equation}
where the constants involved only depend on $\mu$ and $r$. Therefore, \eqref{eq 2} can be proved mimicking the proof of \cite[Theorem~2]{AS}. 
This finishes the proof. 
\end{proof}

\begin{Prf}{\em{ Theorem~\ref{th:schattendugintro}}.}
Part (a) and the equivalence between (i) and (ii) in part (b) follows from Theorem~\ref{th:schattendug}. Moreover, since $\mu\in \DDD$, then $\mug^p\in \DD$.  This together with the fact that $ \frac{\mug(r)^p}{(1-r)^2}$ is a weight implies that
 $ \frac{\mug(r)^p}{(1-r)^2}\in \DD$. So, by hypotheses $ \frac{\mug(r)^p}{(1-r)^2}\in \DDD$. Therefore, by Theorem~\ref{th:BpmuD},  (i) and (ii) are equivalent to the condition $g\in B_p$ and
$$\| V_{\mu,g}\|\asymp \|g\|_{B_{p,\mu}} \asymp \|g\|_{B_p},\quad g\in \H(\D).$$
This finishes the proof.
\end{Prf}
\section{Statements of Declaration}
The authors do not have any financial or non financial interest directly or indirectly related to this article.

\end{document}